\documentclass{birkjour}
\usepackage{amssymb, amsmath, amsthm, graphics, comment, xspace, enumerate}
\usepackage{graphicx}

\allowdisplaybreaks

\newcommand{\NN}{\mathbb N}
\newcommand{\QQ}{\mathbb Q}

\newcommand{\CC}{\mathbb C}
\newcommand{\RR}{\mathbb R}
\newcommand{\ZZ}{\mathbb Z}
\newcommand{\ds}{\displaystyle}
\newcommand{\EE}{\mathcal E}
\newcommand{\DD}{\mathcal D}

\newcommand{\BB}{\mathcal B}

\newcommand{\beq}{\begin{eqnarray}}
\newcommand{\eeq}{\end{eqnarray}}

\newcommand{\beqs}{\begin{eqnarray*}}
\newcommand{\eeqs}{\end{eqnarray*}}
 \newtheorem{thm}{Theorem}[section]
 
 \newtheorem{lem}[thm]{Lemma}
 \newtheorem{prop}[thm]{Proposition}
 \theoremstyle{definition}
 \newtheorem{defn}[thm]{Definition}
 \theoremstyle{remark}
 \newtheorem{rem}[thm]{Remark}
 \newtheorem*{ex}{Example}
 \numberwithin{equation}{section}

\begin{document}

%
%
%
%
%
%
%
%
%

\title[On the solution of the Cauchy problem]{On the solution of the Cauchy problem in the weighted spaces of Beurling 
ultradistributions}

\author[S. Pilipovi\' c]{Stevan Pilipovi\' c}

\address{University of Novi Sad,
Trg D. Obradovi\' ca 4
21000 Novi Sad,
Serbia
}

\email{pilipovic@dmi.uns.ac.rs}
%
\author[B. Prangoski]{Bojan Prangoski}
\address{Department for Mathematics, Faculty of Mechanical
Engineering-Skopje, Karposh 2 b.b., 1000 Skopje, Macedonia}
\email{bprangoski@yahoo.com}

\author[D. Velinov]{Daniel Velinov}
\address{Department for Mathematics, Faculty of Civil
Engineering-Skopje, Partizanski Odredi 24,
P.O. box 560, 1000 Skopje, Macedonia}
\email{velinovd@gf.ukim.edu.mk}


\begin{abstract}
Results of Da Prato and Sinestrari \cite{d81}, on differential operators with non-dense domain but satisfying the Hille--Yosida condition,
are applied in the setting of Beurling weighted spaces of ultradistributions
 $\DD'^{(s)}_{L^p}((0,T)\times U)$, where $U$ is open and bounded set in $\mathbb R^d$.

\end{abstract}

\date{}
\maketitle

\section{Introduction}

Da Prato and Sinestrary \cite{d81}   have studied
the Cauchy problem
\begin{equation}\label{ACP} u'(t)=Au(t)+f(t), u(0)=u_0,\end{equation}
where $A$ is a closed operator in a Banach space $E$ with not
necessarily dense domain in $E$ but satisfying the Hille-Yosida
condition. Here  $u_0\in E$, $f$ is  the $E$-valued continuous or
$L^p-$ function on $[0,T]$. They have considered various classes
of equations and types of solutions illustrating their theory.
Regularity properties of solutions is extended much later in
\cite{S2010}.

Our aim in this paper is to extend the results of \cite{d81} for (\ref{ACP}) to  weighted Schwartz spaces of distributions and Beurling space of  ultradistributions \cite{k91}-\cite{k82}.
Since the weighted Schwartz space $\mathcal D'_{L^p}$ (\cite{SchwartzV1}) can be involved in this theory similarly as Beurling type spaces, and the second ones are more delicate, we focus our investigations to
the Beurling case, more precisely to the space of ultradistributions $\mathcal D'^{(s)}_{L^p}((0,T)\times U))$, $U$ is a bounded domain in  $\mathbb R^d,$ related to the Gevrey sequence $p!^s, s>1$ (see \cite{PilipovicT} for $U=\mathbb R^d$).
In order to apply results of \cite{d81} in this abstract setting,
we study the topological structure of  spaces $\DD^{s}_{L^p,h}(U),
p\in[1,\infty]$ (with a special analysis for $p=\infty$), the
closures of $\DD^{(s)}(U)$ in such spaces, corresponding
projective limits, tensor products, their duals as well as  vector
valued version of such spaces. As a special result, we note that $\DD_{L^p}^{(s)}(U)$ is nuclear for bounded $U$. Also we have that all spaces $\DD_{L^p}^{(s)}(U)$ are isomorphic to $\dot{\BB}^{(s)}(U)$ for bounded $U$. Both assertions do not hold for $U=\mathbb R^d$. The main results of the paper are
related to the structure of quoted spaces. Such preparatory
results
 are needed for the formulation of the Cauchy  problem in this abstract setting and for the application of results in \cite{d81}.
 Our main theorem in the second part of the paper reads:

\begin{thm}
Let $U$ be a bounded domain in $\RR^d$ with smooth boundary and $A(x,\partial_x)$ be a strongly elliptic operator of order $2m$ on $U$.
Then for each $f\in \DD'^{(s)}_{L^p}\left((0,T)\times U\right)$ there exists $u\in \DD'^{(s)}_{L^p}\left((0,T)\times U\right)$ such that $$u'_t+A(x,\partial_x)u=f \mbox{ in } \DD'^{(s)}_{L^p}\left((0,T)\times U\right).$$
\end{thm}

In fact, we first solve (\ref{ACP}) in the space of Banach valued
ultradistributions ${\mathcal D}'^{(s)}_{L^p}(0,T;E)$, i.e. \beqs
\langle \mathbf{u}'(t),\varphi(t)\rangle=A\langle
\mathbf{u}(t),\varphi(t)\rangle+\langle
\mathbf{f}(t),\varphi(t)\rangle,\, \forall\varphi\in
\DD^{(s)}_{L^q}(0,T), \eeqs where $A:D(A)\subseteq E\rightarrow E$
is closed operator which satisfies the Hille-Yosida condition
$$\|(\lambda-\omega)^k R(\lambda:A)^k\|\leq C, \mbox{ for } \lambda>\omega,\,
k\in\ZZ_+.$$ Then, by using the theory that we perviously develop,
we prove the above-mentioned result.

For the background material we mention  \cite{Ph},   \cite{Pazy}, \cite{Lun}, \cite{NS},\cite{S}.
Moreover, we give references for  another approaches to the
abstract Cauchy problem  with non-densely defined $A$  through the
theory of integrated, convoluted, distribution or
ultradistribution semigroups,  \cite{a11}-\cite{cha}, \cite{l115},
\cite{mn}-\cite{LY}, \cite{n181}, \cite{me152}.

  The paper is organized as follows.

The Banach space ${\mathcal D}^{(s)}_{L^p,h}(U)$ and its dual
${\mathcal D}^{'(s)}_{L^p,h}(U)$ are explained in Section 1.
Section 2 is devoted to the Beurling type test spaces ${\mathcal
D}^{(s)}_{L^p}(U)$ and their corresponding duals. In Section 3 we
consider the vector valued ultradistribution spaces ${\mathcal
D}'^{(s)}_{L^p}(U;E)$ and ${\mathcal D}'^{(s)}_{L^p,h}(U;E)$,
where $U$ is a bounded open subset of $\RR^d$. The boundness of
$U$ is important since it implies nuclearity of ${\mathcal
D}^{(s)}_{L^p}(U)$ and ${\mathcal D}^{'(s)}_{L^p}(U)$ which in
turn will imply a very important kernel theorem when $E$ is equal
to ${\mathcal D}'^{(s)}_{L^p}(U)$. In the end of this section we
are particulary interested in the spaces ${\mathcal
D}'^{(s)}_{L^p}(U;E)$ when $E$ is a Banach space. We start Section
4 by defining the Banach space $\tilde{{\mathcal
D}}'^s_{L^p,h}(0,T;E)$ consisting of sequences of Bochner $L^p$
functions with certain growth condition. In this abstract setting
we define the Cauchy problem (\ref{ACP}) and recall from \cite
{d81} two types of solutions of (\ref{ACP}). Then, using the
proof in \cite{d81} we prove the
existence of such solutions in $\tilde{{\mathcal
D}}'^s_{L^p,h}(0,T;E)$ and use this to prove existence of solution
of (\ref{ACP}) in the space of Banach-valued ultradistributions
${\mathcal D}'^{(s)}_{L^p}(0,T;E)$. We apply in Section 5
results of Section 4 for several important instances of $A$ and $E$
considered by Da Prato and Sinestrari in \cite{d81}, but in our
ultradistributional setting. The main part is the proof of the theorem
 that we announced above by using the theory developed
in the Sections 1--3.

\subsection{Preliminaries}

The sets of natural, integer, positive integer, real and complex numbers are denoted by $\NN$, $\ZZ$, $\ZZ_+$, $\RR$, $\CC$. We use the symbols for $x\in \RR^d$: $\langle x\rangle =(1+|x|^2)^{1/2} $,
$D^{\alpha}= D_1^{\alpha_1}\ldots D_d^{\alpha_d},\quad D_j^
{\alpha_j}={i^{-1}}\partial^{\alpha_j}/{\partial x}^{\alpha_j}$, $\alpha=(\alpha_1,\alpha_2,\ldots,\alpha_d)\in\NN^d$.\\
\indent Let $s>1$ and $U\subseteq\RR^d$ be an open set. Following Komatsu \cite{k91}, for a compact set $K\subseteq U$, define $\EE^{s,h}(K)$ as the Banach space (from now on abbreviated as $(B)$-space) of all $\varphi\in \mathcal{C}^{\infty}(U)$ which satisfy $\ds\sup_{\alpha\in\NN^d}\sup_{x\in K}\frac{|D^{\alpha}\varphi(x)|}{h^{|\alpha|}\alpha!^s}<\infty$ and $\DD^{s,h}_K$ as the $(B)$-space of all $\varphi\in \mathcal{C}^{\infty}\left(\RR^d\right)$ with support in $K$, which satisfy $\ds\sup_{\alpha\in\NN^d}\sup_{x\in K}\frac{|D^{\alpha}\varphi(x)|}{h^{\alpha}\alpha!^s}<\infty$. Define the spaces
$$
\EE^{(s)}(U)=\lim_{\substack{\longleftarrow\\ K\subset\subset U}}\lim_{\substack{\longleftarrow\\ h\rightarrow 0}} \EE^{s,h}(K),\,\,\,\,
\EE^{\{s\}}(U)=\lim_{\substack{\longleftarrow\\ K\subset\subset U}}
\lim_{\substack{\longrightarrow\\ h\rightarrow \infty}} \EE^{s,h}(K),
$$
\beqs
\DD^{(s)}_K=\lim_{\substack{\longleftarrow\\ h\rightarrow 0}} \DD^{s,h}_K,\,\,\,\, \DD^{(s)}(U)=\lim_{\substack{\longrightarrow\\ K\subset\subset U}}\DD^{(s)}_K,\\
\DD^{\{s\}}_K=\lim_{\substack{\longrightarrow\\ h\rightarrow \infty}} \DD^{s,h}_K,\,\,\,\, \DD^{\{s\}}(U)=\lim_{\substack{\longrightarrow\\ K\subset\subset U}}\DD^{\{s\}}_K.
\eeqs
The spaces of ultradistributions and ultradistributions with compact support of Beurling and Roumieu type are defined as the strong duals of $\DD^{(s)}(U)$ and $\EE^{(s)}(U)$, resp. $\DD^{\{s\}}(U)$ and $\EE^{\{s\}}(U)$. For the properties of these spaces, we refer to \cite{k91}, \cite{Komatsu2} and \cite{k82}.\\
\indent It is said that $\ds P(\xi ) =\sum _{\alpha \in \NN^d}c_{\alpha } \xi^{\alpha}$, $\xi \in \RR^d$, is an ultrapolynomial of the class $(s)$, resp. $\{s\}$, whenever the coefficients $c_{\alpha }$ satisfy the estimate $|c_{\alpha }|  \leq C L^{|\alpha|}/ \alpha!^s$, $\alpha \in \NN^d$ for some $L > 0$ and $C>0$, resp. for every $L > 0 $ and some $C_{L} > 0$. The corresponding operator  $P(D)=\sum_{\alpha} c_{\alpha}D^{\alpha}$ is an ultradifferential operator of the class $(s)$, resp. $\{s\}$ and they act continuously on $\EE^{(s)}(U)$ and $\DD^{(s)}(U)$, resp. $\EE^{\{s\}}(U)$ and $\DD^{\{s\}}(U)$ and the corresponding spaces of ultradistributions.

\section{ Banach  spaces of weighted ultradistributions}

\subsection{Basic Banach spaces}
Let $U$ be an open subset of $\RR^d$ and $1\leq p\leq \infty$. Let $\DD^{s}_{L^p,h}(U)$ be the space of all $\varphi\in\mathcal{C}^{\infty}(U)$ such that the norm $\ds \left(\sum_{\alpha\in\NN^d}\frac{h^{p|\alpha|}\left\|D^{\alpha}\varphi\right\|_{L^p(U)}^p}{\alpha!^{ps}}\right)^{1/p}$ is finite (with the obvious meaning when $p=\infty$).
One can simply prove:
\begin{lem}
$\DD^{s}_{L^p,h}(U)$ is a $(B)$-space, when $1\leq p \leq \infty$.
\end{lem}

Let$1\leq p\leq \infty$ and $1\leq q\leq \infty$ be such that $\ds \frac{1}{p}+\frac{1}{q}=1$. Let $\DD^{(s)}_{L^p,h}(U)$ denotes the closure of $\DD^{(s)}(U)$ in $\DD^{s}_{L^p,h}(U)$. Denote by $\DD'^{(s)}_{L^p,h}(U)$ the strong dual of $\DD^{(s)}_{L^q,h}(U)$. Then, $\DD'^{(s)}_{L^p,h}(U)$ is continuously injected in $\DD'^{(s)}(U)$, for $1\leq p\leq \infty$. We will denote by $\mathcal{C}_0(U)$ the space of all continuous functions $f$ on $U$ such that for every $\varepsilon>0$ there exists $K\subset\subset U$ such that $|f(x)|< \varepsilon$ when $x\in U\backslash K$. We leave the proof of the next lemma to the reader.

\begin{lem}\label{10}
Let $\varphi\in\DD^{(s)}_{L^{\infty},h}(U)$. Then for every $\varepsilon>0$ there exist $K\subset\subset U$ and $k\in\ZZ_+$ such that
\beqs
\ds \sup_{\alpha\in\NN^d}\sup_{x\in U\backslash K}\frac{h^{|\alpha|}\left|D^{\alpha}\varphi(x)\right|}{\alpha!^s}\leq\varepsilon \mbox{ and } \sup_{|\alpha|\geq k} \frac{h^{|\alpha|}\left\|D^{\alpha}\varphi\right\|_{L^{\infty}(U)}}{\alpha!^s}\leq\varepsilon.
\eeqs
\end{lem}

\subsection{Duals of Banach spaces}

The main goal in this subsection is to give a representation of
the elements of $\DD'^{(s)}_{L^p,h}(U)$, $1\leq p \leq\infty$. In
order to do that, first we will construct a $(B)$-space
which will contain $\DD^{(s)}_{L^p,h}(U)$ as a closed subspace.
It is worth to note that the
main idea of this constructions is due to Komatsu \cite{k91}.\\
\indent For $1\leq p<\infty$ define \beqs
Y_{h,L^p}={\Big\{}(\psi_{\alpha})_{\alpha\in\NN^d}\Big|\,
\psi_{\alpha}\in L^p(U),\,
\left\|(\psi_{\alpha})_{\alpha}\right\|_{Y_{h,L^p}}=\eeqs \beqs
=\left(\sum_{\alpha\in\NN^d}\frac{h^{p|\alpha|}\left\|\psi_{\alpha}\right\|_{L^p(U)}^p}
{\alpha!^{ps}}\right)^{1/p}<\infty{\Big\}}.\eeqs Then one easily
verifies that $Y_{h,L^p}$ is a $(B)$-space, with the norm
$\|\cdot\|_{Y_{h,L^p}}$, for $1\leq p< \infty$. Let $p=\infty$.
Define \beqs
Y_{h,L^{\infty}}=\left\{(\psi_{\alpha})_{\alpha\in\NN^d}\Big|\,
\psi_{\alpha}\in \mathcal{C}_0(U),\,
\lim_{|\alpha|\rightarrow\infty}\frac{h^{|\alpha|}\left\|\psi_{\alpha}\right\|_{L^{\infty}(U)}}
{\alpha!^s}=0\right\}, \eeqs
with the norm $\ds\left\|(\psi_{\alpha})_{\alpha}\right\|_{Y_{h,L^{\infty}}}= \sup_{\alpha\in\NN^d}\frac{h^{|\alpha|}}{\alpha!^s}\|\psi_{\alpha}\|_{L^{\infty}(U)}$. One easily verifies that it is a $(B)$-space.\\
\indent Let $\tilde{U}$ be the disjoint union of countable number
of copies of $U$, one for each $\alpha\in\NN^d$, i.e. $\ds
\tilde{U}=\bigsqcup_{\alpha\in\NN^d}U_{\alpha}$, where
$U_{\alpha}=U$. Equip $\tilde{U}$ with the disjoint union
topology. Then $\tilde{U}$ is Hausdorff locally compact space.
Moreover every open set in $\tilde{U}$ is $\sigma$-compact. For
each $1\leq p<\infty$, one can define a Borel measure $\mu_p$ on
$\tilde{U}$ by
$\ds\mu_p(E)=\sum_{\alpha}\frac{h^{|\alpha|p}}{\alpha!^{ps}}|E\cap
U_{\alpha}|$, for $E$ a Borel subset of $\tilde{U}$, where $|E\cap
U_{\alpha}|$ is the Lebesgue measure of $E\cap U_{\alpha}$. It is
obviously locally finite, $\sigma$-finite and $\mu(K)<\infty$ for
every compact subset $K$ of $\tilde{U}$. By the properties of
$\tilde{U}$ described above, $\mu_p$ is regular (both inner and
outer regular). We obtained that $\mu_p$ is a Radon measure. It
follows that $Y_{h,L^p}$ is exactly $L^p(\tilde{U},\mu_p)$, for
$1\leq p<\infty$. In particular, $Y_{h,L^p}$ is a reflexive
$(B)$-space for $1<p<\infty$. For $p=\infty$, we will prove that
$Y_{h,L^{\infty}}$ is isomorphic to $\mathcal{C}_0(\tilde{U})$.
For $\psi\in\mathcal{C}_0(\tilde{U})$ denote by $\psi_{\alpha}$
the restriction of $\psi$ to $U_{\alpha}$. By the definition of
$\tilde{U}$, $K$ is compact subset of $\tilde{U}$ if and only if
$K\cap U_{\alpha}\neq \emptyset$ for only finitely many
$\alpha\in\NN^d$ and for those $\alpha$, $K\cap U_{\alpha}$ is
compact subset of $U_{\alpha}$. Now, one easily verifies that
$\psi_{\alpha}\in \mathcal{C}_0(U)$ and $\ds
\lim_{|\alpha|\rightarrow\infty}\|\psi_{\alpha}\|_{L^{\infty}(U)}=0$.
Moreover, if $\psi_{\alpha}\in\mathcal{C}_0(U)$, $\alpha\in\NN^d$,
are such that $\ds
\lim_{|\alpha|\rightarrow\infty}\|\psi_{\alpha}\|_{L^{\infty}(U)}=0$
then the function $\psi$ on $\tilde{U}$, defined by
$\psi(x)=\psi_{\alpha}(x)$, when $x\in U_{\alpha}$ is an element
of $\mathcal{C}_0(\tilde{U})$. We obtain that \beqs
\mathcal{C}_0(\tilde{U})=\left\{(\psi_{\alpha})_{\alpha\in\NN^d}\Big|\,
\psi_{\alpha}\in\mathcal{C}_0(U),\, \forall\alpha\in\NN^d,\,
\lim_{|\alpha|\rightarrow\infty}\|\psi_{\alpha}\|_{L^{\infty}(U)}=0\right\}.
\eeqs
Observe that the mapping $(\psi_{\alpha})_{\alpha\in\NN^d}\mapsto (\tilde{\psi}_{\alpha})_{\alpha\in\NN^d}$, where $\ds\tilde{\psi}_{\alpha}=\frac{h^{|\alpha|}}{\alpha!^s}\psi_{\alpha}$, is an isometry from $Y_{h,L^{\infty}}$ onto $\mathcal{C}_0(\tilde{U})$. For the purpose of the next proposition we will denote by $\iota$ the inverse mapping of this isometry, i.e. $\iota:\mathcal{C}_0(\tilde{U})\rightarrow Y_{h,L^{\infty}}$.\\
\indent Note that $\DD^{(s)}_{L^p,h}(U)$ can be identified with a closed subspace of $Y_{h,L^p}$ by the mapping $\varphi\mapsto ((-D)^{\alpha}\varphi)_{\alpha\in\NN^d}$. This is obvious for $1\leq p<\infty$ and for $p=\infty$ it follows from Lemma \ref{10}. Since $Y_{h,L^p}$ is reflexive for $1<p<\infty$ so is $\DD^{(s)}_{L^p,h}(U)$ as a closed subspace of a reflexive $(B)$-space.\\

Observe that spaces $L^p(U)$, for $1\leq p\leq \infty$, resp, $\left(\mathcal{C}_0(U)\right)'$, are continuously injected into $\DD'^{(s)}_{L^p,h}(U)$, resp. $\DD'^{(s)}_{L^1,h}(U)$. For $\alpha\in\NN^d$ and $F\in L^p(U)$, resp. $F\in\left(\mathcal{C}_0(U)\right)'$, we define $D^{\alpha}F\in \DD'^{(s)}_{L^p,h}(U)$, resp. $D^{\alpha}F\in \DD'^{(s)}_{L^1,h}(U)$, by
\beqs
\langle D^{\alpha} F,\varphi\rangle&=& \int_U F(x) (-D)^{\alpha}\varphi(x)dx,\, \varphi\in \DD^{(s)}_{L^q,h}(U), \mbox{ resp.}\\
\langle D^{\alpha} F,\varphi\rangle&=& \int_U (-D)^{\alpha}\varphi(x)dF,\, \varphi\in \DD^{(s)}_{L^{\infty},h}(U).
\eeqs
It is easy to verify that $D^{\alpha}F$ is well defined element of $\DD'^{(s)}_{L^p,h}(U)$, resp. $\DD'^{(s)}_{L^1,h}(U)$, and in fact it is equal to its ultradistributional derivative when we regard $F$ as an element of $\DD'^{(s)}(U)$.

\begin{prop}\label{20}
Let $1< p\leq\infty$. For every $T\in\DD'^{(s)}_{L^p,h}(U)$, there exist $C>0$ and $F_{\alpha}\in L^p(U)$, $\alpha\in\NN^d$, such that
\beq\label{21}
\left(\sum_{\alpha\in\NN^d} \frac{\alpha!^{ps}}{h^{|\alpha|p}}\|F_{\alpha}\|_{L^p(U)}^p\right)^{1/p}\leq C \mbox{ and } T=\sum_{|\alpha|=0}^{\infty} D^{\alpha}F_{\alpha}.
\eeq
When $p=1$, for every $T\in\DD'^{(s)}_{L^1,h}(U)$, there exist $C>0$ and Radon measures $F_{\alpha}\in \left(\mathcal{C}_0(U)\right)'$, $\alpha\in\NN^d$, such that
\beq\label{22}
\sum_{\alpha\in \NN^d} \frac{\alpha!^s}{h^{|\alpha|}}\|F_{\alpha}\|_{\left(\mathcal{C}_0(U)\right)'}\leq C \mbox{ and } T=\sum_{|\alpha|=0}^{\infty} D^{\alpha}F_{\alpha}.
\eeq
Moreover, if $B$ is a bounded subset of $\DD'^{(s)}_{L^p,h}(U)$, then there exists $C>0$ independent of $T\in B$ and for each $T\in B$ there exist $F_{\alpha}\in L^p(U)$, $\alpha\in\NN^d$, for $1<p\leq \infty$, resp. $F_{\alpha}\in \left(\mathcal{C}_0(U)\right)'$, $\alpha\in\NN^d$, for $p=1$, such that (\ref{21}), resp. (\ref{22}), holds.\\
\indent If $F_{\alpha}\in L^p(U)$, $\alpha\in\NN^d$, for $1< p\leq\infty$, resp. $F_{\alpha}\in \left(\mathcal{C}_0(U)\right)'$, $\alpha\in\NN^d$, for $p=1$, are such that $\ds\left(\sum_{\alpha\in\NN^d} \frac{\alpha!^{ps}}{h^{|\alpha|p}}\|F_{\alpha}\|_{L^p(U)}^p\right)^{1/p}<\infty$, for $1< p\leq\infty$, resp. $\ds \sum_{\alpha\in \NN^d} \frac{\alpha!^s}{h^{|\alpha|}}\|F_{\alpha}\|_{\left(\mathcal{C}_0(U)\right)'}<\infty$, for $p=1$, then the series $\ds\sum_{|\alpha|=0}^{\infty}D^{\alpha}F_{\alpha}$ converges absolutely in $\DD'^{(s)}_{L^p,h}(U)$, resp. $\DD'^{(s)}_{L^1,h}(U)$.
\end{prop}

\begin{proof} Let $Y_{h,L^q}$ be as in the above discussion. Extend $T$ by the Hahn-Banach theorem to a continuous functional on $Y_{h,L^q}$ and denote it again by $T$, for $1\leq q\leq \infty$. For $q=\infty$, $\tilde{T}=T\circ\iota$ is a functional on $\mathcal{C}_0(\tilde{U})$. Then, for $1< p\leq\infty$, there exists $g\in L^p(\tilde{U},\mu_q)$ such that $\ds T\left((\psi_{\alpha})_{\alpha\in\NN^d}\right)=\int_{\tilde{U}}(\psi_{\alpha})_{\alpha\in\NN^d}gd\mu_q$, $(\psi_{\alpha})_{\alpha\in\NN^d}\in Y_{h,L^q}$. For $p=1$, there exists $g\in \left(\mathcal{C}_0(\tilde{U})\right)'$ such that $\ds \tilde{T}(\psi)=\int_{\tilde{U}}\psi dg$, for $\psi\in \mathcal{C}_0(\tilde{U})$. Hence, for $(\psi_{\alpha})_{\alpha\in\NN^d}\in Y_{h,L^{\infty}}$, we have
\beqs
T\left((\psi_{\alpha})_{\alpha\in\NN^d}\right)=\tilde{T}\left((\tilde{\psi}_{\alpha})_{\alpha\in\NN^d}\right)= \int_{\tilde{U}}(\tilde{\psi}_{\alpha})_{\alpha\in\NN^d}dg,
\eeqs
where $\ds(\tilde{\psi}_{\alpha})_{\alpha}=\iota^{-1}\left((\psi_{\alpha})_{\alpha}\right)= \left(\frac{h^{|\alpha|}}{\alpha!^s}\psi_{\alpha}\right)_{\alpha}$. Put $\ds F_{\alpha}=\frac{h^{|\alpha|q}}{\alpha!^{qs}}g_{|U_{\alpha}}$, for $1\leq q<\infty$. For $q=\infty$, put $\ds F_{\alpha}=\frac{h^{|\alpha|}}{\alpha!^s}g_{|U_{\alpha}}$. Then $F_{\alpha}\in L^p(U)$, for $1\leq q<\infty$, respectively $F_{\alpha}\in \left(\mathcal{C}_0(U)\right)'$ for $q=\infty$. Moreover, for $1< q<\infty$,
\beqs
\sum_{\alpha\in\NN^d}\frac{\alpha!^{ps}}{h^{|\alpha|p}}\|F_{\alpha}\|_{L^p(U)}^p= \sum_{\alpha\in\NN^d}\frac{h^{|\alpha|q}}{\alpha!^{qs}}\left\|g_{|U_{\alpha}}\right\|_{L^p(U)}^p =\|g\|_{L^p(\tilde{U},\mu_q)}^p<\infty.
\eeqs
Also, it is easy to verify that, for $q=1$, $\ds \sup_{\alpha}\frac{\alpha!^s}{h^{|\alpha|}}\|F_{\alpha}\|_{L^{\infty}(U)}=\|g\|_{L^{\infty}(\tilde{U},\mu_1)}<\infty$. For $q=\infty$ we have
\beqs
\sum_{\alpha\in\NN^d}\frac{\alpha!^s}{h^{|\alpha|}}\|F_{\alpha}\|_{\left(\mathcal{C}_0(U)\right)'}= \sum_{\alpha\in\NN^d}\left\|g_{|U_{\alpha}}\right\|_{\left(\mathcal{C}_0(U)\right)'}=\|g\|_{\left(\mathcal{C}_0(\tilde{U})\right)'} <\infty,
\eeqs
where in the second equality we used that $\left\|g_{|U_{\alpha}}\right\|_{\left(\mathcal{C}_0(U)\right)'}= \left|g_{|U_{\alpha}}\right|(U_{\alpha})=|g|(U_{\alpha})$ (we denote by $|g|$ the total variation of the measure $g$ and similarly for $g_{|U_{\alpha}}$). Moreover
\beqs
T\left((\psi)_{\alpha\in\NN^d}\right)=\sum_{\alpha\in\NN^d}\int_{U}\psi_{\alpha}(x)F_{\alpha}(x)dx,
\eeqs
for $1\leq q<\infty$. For $q=\infty$ we have
\beqs
T\left((\psi_{\alpha})_{\alpha\in\NN^d}\right)=\int_{\tilde{U}}(\tilde{\psi}_{\alpha})_{\alpha\in\NN^d}dg=\sum_{\alpha\in\NN^d} \frac{\alpha!^s}{h^{|\alpha|}}\int_U\tilde{\psi}_{\alpha}dF_{\alpha}=\sum_{\alpha\in\NN^d}\int_U\psi_{\alpha}dF_{\alpha}.
\eeqs
So, for $1\leq q<\infty$, if $\varphi\in \DD^{(s)}_{L^q,h}(U)$, we obtain
\beqs
\langle T,\varphi\rangle=\sum_{\alpha\in\NN^d}\int_{U}(-D)^{\alpha}\varphi(x)F_{\alpha}(x)dx= \sum_{\alpha\in\NN^d}\langle D^{\alpha}F_{\alpha},\varphi\rangle.
\eeqs
Similarly, $\langle T,\varphi\rangle=\sum_{\alpha}\langle D^{\alpha}F_{\alpha},\varphi\rangle$ when $q=\infty$. Moreover, by these calculations, it follows that for $1\leq q<\infty$
\beqs
\sum_{\alpha\in\NN^d}\left|\langle D^{\alpha}F_{\alpha},\varphi\rangle\right|\leq \left(\sum_{\alpha\in\NN^d} \frac{\alpha!^{ps}}{h^{|\alpha|p}}\|F_{\alpha}\|_{L^p(U)}^p\right)^{1/p} \left(\sum_{\alpha\in\NN^d}\frac{h^{|\alpha|q}\left\|D^{\alpha}\varphi\right\|_{L^q(U)}^q} {\alpha!^{qs}}\right)^{1/q}.
\eeqs
Hence the partial sums of $\sum_{\alpha}D^{\alpha}F_{\alpha}$ converge absolutely in $\DD'^{(s)}_{L^p,h}(U)$, when $1<p\leq \infty$. When $p=1$, the proof that the partial sums of $\sum_{\alpha} D^{\alpha} F_{\alpha}$ converge absolutely in $\DD'^{(s)}_{L^1,h}(U)$ is similar and we omit it. If $B$ is a bounded subset of $\DD'^{(s)}_{L^p,h}(U)$, by the Hahn-Banach theorem it can be extended to a bounded set $B_1$ in $Y'_{h,L^q}$, for $1\leq q<\infty$, resp. to a bounded set $B_1$ in $\mathcal{C}_0(\tilde{U})$ for $q=\infty$ ($\iota$ is an isometry). Hence, there exists $C>0$ independent of $T\in B_1$ and for each $T\in B_1$ there exists $g\in L^p(\tilde{U},\mu_q)$, for $1<p\leq\infty$, resp. $g\in\left(\mathcal{C}_0(\tilde{U})\right)'$, for $p=1$, such that $\|g\|_{L^p(\tilde{U})}\leq C$, resp. $\|g\|_{\left(\mathcal{C}_0(\tilde{U})\right)'}\leq C$. If we define $F_{\alpha}$ as above one obtains (\ref{21}), resp. (\ref{22}), with the desired uniform estimate independent of $T\in B$.\\
\indent The last part of the proposition is easy and we omit it.
\end{proof}

\section{Ultradistribution spaces}

\subsection{Beurling type test spaces}
For $1\leq p\leq \infty$, we define locally convex spaces (from now on abbreviated as l.c.s.) $\ds \BB^{(s)}_{L^p}(U)=\lim_{\substack{\longleftarrow\\ h\rightarrow \infty}}\DD^{s}_{L^p,h}(U)$. Then $\BB^{(s)}_{L^p}(U)$ is a $(F)$-space. Denote by $\DD^{(s)}_{L^p}(U)$ the closure of $\DD^{(s)}(U)$ in $\BB^{(s)}_{L^p}(U)$ for $1\leq p<\infty$ and $\dot{\BB}^{(s)}(U)$ the closure of $\DD^{(s)}(U)$ in $\BB^{(s)}_{L^{\infty}}(U)$. Hence, when $U=\RR^d$, these spaces coincide with the spaces $\DD^{(s)}_{L^p}(\RR^d)$, for $1\leq p<\infty$, resp. $\dot{\BB}^{(s)}$ defined in \cite{PilipovicT}. All of these spaces are $(F)$-spaces as well as  $\ds X_{L^p}=\lim_{\substack{\longleftarrow\\ h\rightarrow \infty}} \DD^{(s)}_{L^p,h}(U)$ $1\leq p\leq \infty$.

\begin{lem} Let $X_{L^p}$ be as above and $1\leq p\leq \infty$.
\begin{itemize}
\item[$i)$] $\DD^{(s)}(U)$ is dense in $X_{L^p}$.
\item[$ii)$] $X_{L^p}$ is a closed subspace of $\BB^{(s)}_{L^p}(U)$ and the topology of $X_{L^p}$ is the same as the induced one from $\BB^{(s)}_{L^p}(U)$. Hence $X_{L^p}$ and $\DD^{(s)}_{L^p}(U)$, for $1\leq p<\infty$, resp. $X_{L^{\infty}}$ and $\dot{\BB}^{(s)}(U)$ when $p=\infty$, are isomorphic l.c.s.
\end{itemize}
\end{lem}

\begin{proof} Since $\DD^{(s)}(U)$ is dense in each $\DD^{(s)}_{L^p,h}(U)$ it follows that $\DD^{(s)}(U)\subseteq X_{L^p}$ and it is dense in $X_{L^p}$. The proof of $i)$ is complete. To prove $ii)$ note that $X_{L^p}\subseteq \BB^{(s)}_{L^p}(U)$. Let $\varphi_j$, $j\in\NN$, be a sequence in $X_{L^p}$ which converges to $\varphi\in \BB^{(s)}_{L^p}(U)$ in the topology of $\BB^{(s)}_{L^p}(U)$. Then $\varphi_j$ converges to $\varphi$ in $\DD^{s}_{L^p,h}(U)$ for each $h$. But $\varphi_j\in \DD^{(s)}_{L^p,h}(U)$, $j\in\NN$ and $\DD^{(s)}_{L^p,h}(U)$ is a closed subspace of $\DD^{s}_{L^p,h}(U)$ with the same topology. It follows that $\varphi\in \DD^{(s)}_{L^p,h}(U)$ and $\varphi_j$ converges to $\varphi$ in $\DD^{(s)}_{L^p,h}(U)$ for each $h$. Hence $\varphi\in X_{L^p}$. Moreover, since the inclusion $X_{L^p}\rightarrow \BB^{(s)}_{L^p}(U)$ is obviously continuous and $X_{L^p}$ and $\BB^{(s)}_{L^p}(U)$ are $(F)$-spaces and the image of $X_{L^p}$ under the inclusion is closed subspace of $\BB^{(s)}_{L^p}(U)$ by the open mapping theorem it follows that $X_{L^p}$ is isomorphic with its image under this inclusion (isomorphic as l.c.s.).
\end{proof}

By the above lemma we obtain that $\ds \DD^{(s)}_{L^p}(U)=\lim_{\substack{\longleftarrow\\ h\rightarrow \infty}} \DD^{(s)}_{L^p,h}(U)$, for $1\leq p<\infty$ and $\ds \dot{\BB}^{(s)}(U)=\lim_{\substack{\longleftarrow\\ h\rightarrow \infty}} \DD^{(s)}_{L^{\infty},h}(U)$, for $p=\infty$ and the projective limits are reduced. For $1< p\leq \infty$, denote by $\DD'^{(s)}_{L^p}(U)$ the strong dual of $\DD^{(s)}_{L^q}(U)$. Denote by $\DD'^{(s)}_{L^1}(U)$ the strong dual of $\dot{\BB}^{(s)}(U)$. Since $\DD^{(s)}(U)$ is continuously and densely injected into $\DD^{(s)}_{L^q}(U)$, for $1\leq q<\infty$ and into $\dot{\BB}^{(s)}(U)$, $\DD'^{(s)}_{L^p}(U)$ are continuously injected into $\DD'^{(s)}(U)$, for $1\leq p\leq \infty$. One easily verifies that ultradifferential operators of class $(s)$ act continuously on $\DD^{(s)}_{L^p}(U)$, for $1\leq p<\infty$ and on $\dot{\BB}^{(s)}(U)$. Hence they act continuously on $\DD'^{(s)}_{L^p}(U)$, for $1\leq p\leq \infty$. For $1<p<\infty$, since all $\DD^{(s)}_{L^p,h}(U)$ are reflexive $(B)$-spaces, the inclusion $\DD^{(s)}_{L^p,h_2}(U)\rightarrow \DD^{(s)}_{L^p,h_1}(U)$, for $h_2>h_1$ is weakly compact mapping, hence $\DD^{(s)}_{L^p}(U)$ is a $(FS^*)$-space, in particular it is reflexive.\\
\indent From now on we suppose that $U$ is bounded open set in $\RR^d$.

\begin{prop}\label{50}
Let $1\leq p<\infty$ and $h_1>h$. We have the continuous inclusions $\DD^{(s)}_{L^{\infty},h_1}(U)\rightarrow \DD^{(s)}_{L^p,h}(U)$ and $\DD^{(s)}_{L^p,2^sh}(U)\rightarrow \DD^{(s)}_{L^{\infty},h}(U)$. In particular, the spaces $\DD^{(s)}_{L^p}(U)$, $1\leq p<\infty$ and $\dot{\BB}^{(s)}(U)$ are isomorphic among each other.
\end{prop}

\begin{proof} Let $1\leq p<\infty$ and $\varphi\in\DD^{(s)}_{L^p,h}(U)$. It is obvious that for each $\alpha\in\NN^d$, $D^{\alpha}\varphi\in W^{m,p}_0(U)$, for any $m\in\ZZ_+$. Hence, by the Sobolev imbedding theorem it follows that for each $\alpha\in\NN^d$, $D^{\alpha}\varphi$ extends to a uniformly continuous function on $\overline{U}$. Now, let $\varphi\in\DD^{s}_{L^{\infty},h_1}(U)$. Then
\beqs
\left(\sum_{\alpha\in\NN^d}\frac{h^{p|\alpha|}\left\|D^{\alpha}\varphi\right\|_{L^p(U)}^p}{\alpha!^{ps}}\right)^{1/p} &\leq& |U|^{1/p} \left(\sum_{\alpha\in\NN^d}\frac{h^{p|\alpha|}h_1^{p|\alpha|}\left\|D^{\alpha}\varphi\right\|_{L^{\infty}(U)}^p} {h_1^{p|\alpha|}\alpha!^{ps}}\right)^{1/p}\\
&\leq& C|U|^{1/p} \sup_{\alpha\in\NN^d}\frac{h_1^{|\alpha|}\left\|D^{\alpha}\varphi\right\|_{L^{\infty}(U)}}{\alpha!^s}.
\eeqs
We obtain that the inclusion $\DD^{s}_{L^{\infty},h_1}(U)\rightarrow\DD^{s}_{L^p,h}(U)$ is continuous. Moreover, if $\varphi\in \DD^{(s)}_{L^{\infty},h_1}(U)$, then there exist $\varphi_j\in\DD^{(s)}(U)$, $j\in\ZZ_+$, such that $\varphi_j\rightarrow \varphi$, as $j\rightarrow \infty$, in $\DD^{s}_{L^{\infty},h_1}(U)$. But then $\varphi_j\rightarrow \varphi$, as $j\rightarrow \infty$, in $\DD^{s}_{L^p,h}(U)$. Hence, $\DD^{(s)}_{L^{\infty},h_1}(U)$ is continuously injected into $\DD^{(s)}_{L^p,h}(U)$. It follows that for each $\varphi\in\DD^{(s)}_{L^{\infty},h_1}(U)$, $\alpha\in\NN^d$, $D^{\alpha}\varphi$ can be extended to a uniformly continuous function on $\overline{U}$. Let $\varphi\in \DD^{(s)}_{L^p,2^sh}(U)$. Fix $m\in\ZZ_+$, such that $mp>d$. Denote by $\ds C_1=\max_{|\alpha|\leq m} \alpha!^s/h^{|\alpha|}$. By the Sobolev imbedding theorem we have
\beqs
\frac{h^{|\beta|}\|D^{\beta}\varphi\|_{L^{\infty}(U)}}{\beta!^s}&\leq& C'\frac{h^{|\beta|}}{\beta!^s}\left(\sum_{|\alpha|\leq m} \|D^{\alpha+\beta}\varphi\|_{L^p(U)}^p\right)^{1/p}\\
&\leq& C'\left(\sum_{|\alpha|\leq m} \frac{h^{(|\alpha|+|\beta|)p}\alpha!^{ps}}{\beta!^{ps}\alpha!^{ps} h^{|\alpha|p}} \|D^{\alpha+\beta}\varphi\|_{L^p(U)}^p\right)^{1/p}\\
&\leq& C'C_1\left(\sum_{|\alpha|\leq m} \frac{(2^sh)^{(|\alpha|+|\beta|)p}}{(\alpha+\beta)!^{ps}} \|D^{\alpha+\beta}\varphi\|_{L^p(U)}^p\right)^{1/p}\\
&\leq& C'C_1\left(\sum_{\gamma \in \NN^d} \frac{(2^sh)^{|\gamma|p}}{\gamma!^{ps}} \|D^{\gamma}\varphi\|_{L^p(U)}^p\right)^{1/p}.
\eeqs
We obtain that $\DD^{(s)}_{L^p,2^sh}(U)$ is continuously injected in $\DD^{s}_{L^{\infty},h}(U)$. Moreover, if $\varphi_j\in\DD^{(s)}(U)$, $j\in\ZZ_+$, are such that $\varphi_j\rightarrow \varphi$, when $j\rightarrow \infty$, in $\DD^{(s)}_{L^p,2^sh}(U)$, then $\varphi_j\rightarrow \varphi$, when $j\rightarrow \infty$, in $\DD^{s}_{L^{\infty},h}(U)$. Hence, $\DD^{(s)}_{L^p,2^sh}(U)$ is continuously injected into $\DD^{(s)}_{L^{\infty},h}(U)$.
\end{proof}

Proposition \ref{50} implies that, we no longer need to distinguish the spaces $\DD^{(s)}_{L^p}(U)$ since they are all isomorphic to $\dot{\BB}^{(s)}(U)$. Hence their duals are all isomorphic to $\DD'^{(s)}_{L^1}(U)$.

\begin{prop}\label{60}
Let $U$ be bounded open subset of $\RR^d$.
\begin{itemize}
\item[$i)$] Let $h>0$ be fixed. Every element $\varphi$ of $\DD^{(s)}_{L^p,h}(U)$ for $1\leq p\leq\infty$, can be extended to $\mathcal{C}^{\infty}$ function on $\RR^d$ with support in $\overline{U}$. Moreover $\DD^{(s)}_{L^{\infty},h}(U)$ can be identified with a closed subspace of $\DD^{s,h}_{\overline{U}}$;
\item[$ii)$] $\dot{\BB}^{(s)}(U)$ can be identified with a closed subspace of $\DD^{(s)}_{\overline{U}}$;
\item[$iii)$] $\dot{\BB}^{(s)}(U)$ is a nuclear $(FS)$-spaces. Moreover, in the representation $\ds\dot{\BB}^{(s)}(U)=\lim_{\substack{\longleftarrow\\ h\rightarrow \infty}} \DD^{(s)}_{L^{\infty},h}(U)$, the linking inclusions in the projective limit $\DD^{(s)}_{L^{\infty},h_1}(U)\rightarrow\DD^{(s)}_{L^{\infty},h}(U)$ are compact for $h_1>h$.
\end{itemize}
\end{prop}

\begin{proof} To prove the first part of $i)$, note that by Proposition \ref{50}, $\DD^{(s)}_{L^p,h}(U)$ is continuously injected into $\DD^{(s)}_{L^{\infty},h/2^s}(U)$. Hence it is enough to prove it for $\DD^{(s)}_{L^{\infty},h}(U)$. Let $\varphi\in\DD^{(s)}_{L^{\infty},h}(U)$. Then there exist $\varphi_j\in\DD^{(s)}(U)$, $j\in\ZZ_+$, such that $\varphi_j\rightarrow\varphi$, as $j\rightarrow\infty$ in $\DD^{(s)}_{L^{\infty},h}(U)$. So for $\varepsilon>0$ there exists $j_0\in\ZZ_+$ such that for $j,k\geq j_0$, $j,k\in\ZZ_+$, we have $\ds\sup_{\alpha\in\NN^d}\frac{h^{|\alpha|}\left\|D^{\alpha}\varphi_k-D^{\alpha}\varphi_j\right\|_{L^{\infty}(U)}}{\alpha!^s}\leq \varepsilon$. Since all $\varphi_j$, $j\in\ZZ_+$, have compact support in $U$ and $\DD^{(s)}(U)\subseteq \DD^{s,h}_{\overline{U}}$ we obtain that
\beqs
\sup_{\alpha\in\NN^d}\frac{h^{|\alpha|}\left\|D^{\alpha}\varphi_k-D^{\alpha}\varphi_j\right\|_{L^{\infty}\left(\RR^d\right)}} {\alpha!^s}\leq \varepsilon
\eeqs
for all $j,k\geq j_0$, $j,k\in\ZZ_+$. Hence, $\varphi_j$ is a Cauchy sequence in the $(B)$-space $\DD^{s,h}_{\overline{U}}$ so it must converge to an element $\psi\in\DD^{s,h}_{\overline{U}}$. Hence $\psi(x)=\varphi(x)$, when $x\in U$ and obviously $\psi(x)=0$ when $x\in\RR^d\backslash U$ (since all $\varphi_j$, $j\in\ZZ_+$, have compact support in $U$). This proofs the first part of $i)$. To prove the second part, consider the mapping $\varphi\mapsto \widetilde{\varphi}$, $\DD^{(s)}_{L^{\infty},h}(U)\rightarrow \DD^{s,h}_{\overline{U}}$, where $\widetilde{\varphi}(x)=\varphi(x)$, when $x\in U$ and $\widetilde{\varphi}(x)=0$, when $x\in\RR^d\backslash U$. By the above discussion, this is well defined mapping. Moreover, one easily sees that it is an isometry, which completes the proof of $i)$. Observe that $ii)$ follows from $i)$ since $\ds\dot{\BB}^{(s)}(U)=\lim_{\substack{\longleftarrow\\ h\rightarrow \infty}} \DD^{(s)}_{L^{\infty},h}(U)$ and $\ds\DD^{(s)}_{\overline{U}}=\lim_{\substack{\longleftarrow\\ h\rightarrow \infty}} \DD^{s,h}_{\overline{U}}$. The first part of $iii)$ follows from $ii)$ since $\dot{\BB}^{(s)}(U)$ is a closed subspace of the nuclear $(FS)$-space $\DD^{(s)}_{\overline{U}}$ (Komatsu in \cite{k91} proves the nuclearity of $\DD^{(s)}_{\overline{U}}$ when $\overline{U}$ is regular compact set, but the proof is valid for general $\overline{U}$; the regularity of $\overline{U}$ is used by Komatsu \cite{k91} for the definition and nuclearity of $\EE^{(s)}(\overline{U})$). For the second part, by Proposition 2.2 of \cite{k91} the inclusion $\DD^{s,h_1}_{\overline{U}}\rightarrow\DD^{s,h}_{\overline{U}}$ is compact. Since $\DD^{(s)}_{L^{\infty},h_1}(U)$, resp. $\DD^{(s)}_{L^{\infty},h}(U)$, is closed subspace of $\DD^{s,h_1}_{\overline{U}}$, resp. $\DD^{s,h}_{\overline{U}}$, we obtain the compactness of the inclusion under consideration.
\end{proof}

\subsection{Weighted Beurling spaces of ultradistributions }
\begin{prop}\label{330}
Let $T\in\DD'^{(s)}_{L^1}(U)$. For every $1\leq p\leq \infty$ there exist $h,C>0$ and $F_{\alpha}\in \mathcal{C}(\overline{U})$, $\alpha\in\NN^d$, such that
\beq\label{331}
\left(\sum_{\alpha\in\NN^d} \frac{\alpha!^{ps}}{h^{|\alpha|p}}\|F_{\alpha}\|_{L^{\infty}(U)}^p\right)^{1/p}\leq C \mbox{ and } T=\sum_{\alpha\in\NN^d}D^{\alpha}F_{\alpha},
\eeq
where the last series converges absolutely in $\DD'^{(s)}_{L^1}(U)$. Moreover, if $B$ is a bounded subset of $\DD'^{(s)}_{L^1}(U)$ and $1\leq p\leq \infty$, then there exist $h,C>0$ independent of $T\in B$ and for each $T\in B$ there exist $F_{\alpha}\in \mathcal{C}(\overline{U})$, $\alpha\in\NN^d$, such that (\ref{331}) holds.\\
\indent Conversely, for $1\leq p\leq \infty$, if $F_{\alpha}\in L^p(U)$, $\alpha\in\NN^d$, are such that $\ds\left(\sum_{\alpha\in\NN^d} \frac{\alpha!^{ps}}{h^{|\alpha|p}}\|F_{\alpha}\|_{L^p}^p\right)^{1/p}<\infty$ for some $h>0$ then the series $\ds\sum_{|\alpha|=0}^{\infty}D^{\alpha} F_{\alpha}$ converges absolutely in $\DD'^{(s)}_{L^p,h}(U)$ and hence also in $\DD'^{(s)}_{L^1}(U)$.
\end{prop}

\begin{proof} We will prove first the second part of the proposition. If $F_{\alpha}\in L^p(U)$, $\alpha\in\NN^d$, are as above, the absolute convergence of $\ds\sum_{|\alpha|=0}^{\infty}D^{\alpha} F_{\alpha}$ in $\DD'^{(s)}_{L^p,h}(U)$ follows by proposition \ref{20} for $1<p\leq \infty$ and can be easily verified for $p=1$. By Proposition \ref{50}, $\dot{\BB}^{(s)}(U)$ is continuously and densely injected into $\DD^{(s)}_{L^q,h}(U)$, where $q$ is the conjugate of $p$, i.e. $p^{-1}+q^{-1}=1$ (the part about the denseness follows from the fact that $\DD^{(s)}(U)\subseteq \dot{\BB}^{(s)}(U)$ is dense in $\DD^{(s)}_{L^q,h}(U)$). Hence $\DD'^{(s)}_{L^p,h}(U)$ is continuously injected into $\DD'^{(s)}_{L^1}(U)$ and we obtain that $\ds\sum_{|\alpha|=0}^{\infty}D^{\alpha} F_{\alpha}$ converges absolutely in $\DD'^{(s)}_{L^1}(U)$.\\
\indent To prove the first part, we fix $1< p \leq \infty$ and let $q$ to be the conjugate of $p$. Since $\ds \dot{\BB}^{(s)}(U)=\lim_{\substack{\longleftarrow\\ h\rightarrow \infty}} \DD^{(s)}_{L^{\infty},h}(U)$ and the projective limit is reduced with compact linking mappings (cf. Proposition \ref{60}), $\ds \DD'^{(s)}_{L^1}(U)=\lim_{\substack{\longrightarrow\\ h\rightarrow \infty}} \DD'^{(s)}_{L^1,h}(U)$ as l.c.s., where the inductive limit is injective with compact linking mappings. If $B$ is bounded subset of $\DD'^{(s)}_{L^1}(U)$ there exists $h_1>0$ such that $B\subseteq \DD'^{(s)}_{L^1,h}(U)$ and is bounded there. By proposition \ref{50}, if we take $h=2^sh_1$, $\DD^{(s)}_{L^q,h}(U)$ is continuously injected into $\DD^{(s)}_{L^{\infty},h_1}(U)$. Obviously, $\DD^{(s)}_{L^q,h}(U)$ is dense in $\DD^{(s)}_{L^{\infty},h_1}(U)$ (since $\DD^{(s)}(U)$ is). We obtain that $\DD'^{(s)}_{L^1,h_1}(U)$ is continuously injected into $\DD'^{(s)}_{L^p,h}(U)$. Hence $B$ is a bounded subset of $\DD'^{(s)}_{L^p,h}(U)$. Now, by Proposition \ref{20}, for each $T\in B$ there exist $\tilde{F}_{\alpha}\in L^p(U)$, $\alpha\in\NN^d$, such that
\beqs
\left(\sum_{\alpha\in\NN^d} \frac{\alpha!^{ps}}{h^{p|\alpha|}}\|\tilde{F}_{\alpha}\|^p_{L^p(U)}\right)^{1/p}\leq C' \mbox{ and } T=\sum_{\alpha\in\NN^d}D^{\alpha}\tilde{F}_{\alpha}
\eeqs
and the constant $C'$ is the same for all $T\in B$. Let $L(x)\in\mathcal{C}\left(\RR^d\right)$ be a fundamental solution of $\Delta^d L=\delta$ ($\Delta$ is the Laplacian). Define $\ds G_{\alpha}(x)=\int_U L(x-y)\tilde{F}_{\alpha}(y)dy$, $\alpha\in\NN^d$. Obviously $G_{\alpha}\in\mathcal{C}(\overline{U})$, $\alpha\in\NN^d$ and $\|G_{\alpha}\|_{L^{\infty}(U)}\leq C_1 \|\tilde{F}_{\alpha}\|_{L^p(U)}$, for all $\alpha\in\NN^d$. Hence $\ds \left(\sum_{\alpha\in\NN^d} \frac{\alpha!^{ps}}{h^{p|\alpha|}}\|G_{\alpha}\|_{L^{\infty}(U)}^p\right)^{1/p}\leq C_2$ and $C_2$ is independent of $T\in B$. Let $\Delta^d=\sum_{\beta}c_{\beta}D^{\beta}$ and define $\ds F_{\alpha}=\sum_{\beta\leq \alpha}c_{\beta}G_{\alpha-\beta}$, $\alpha\in\NN^d$. The obviously $F_{\alpha}\in\mathcal{C}(\overline{U})$ for all $\alpha\in\NN^d$. Note that $c_{\beta}\neq 0$ only for finitely many $\beta\in\NN^d$. Put $\ds C_3=\sum_{\beta}\frac{\beta!^s}{h^{|\beta|}}\left|c_{\beta}\right|$. Then\\
\\
$\ds\left(\sum_{\alpha\in\NN^d} \frac{\alpha!^{ps}}{(2^{s+1}h)^{|\alpha|p}}\|F_{\alpha}\|_{L^{\infty}(U)}^p\right)^{1/p}$
\beqs
&\leq& \left(\sum_{\alpha\in\NN^d} \frac{1}{2^{|\alpha|p}} \left(\sum_{\beta\leq\alpha} \frac{(\alpha-\beta)!^s \beta!^s}{h^{|\alpha|-|\beta|}h^{|\beta|}}\left|c_{\beta}\right|\|G_{\alpha-\beta}\|_{L^{\infty}(U)}\right)^p \right)^{1/p}\\
&\leq& C_2C_3\left(\sum_{\alpha\in\NN^d}\frac{1}{2^{|\alpha|p}}\right)^{1/p}
\eeqs
and the last is independent of $T\in B$. Now one easily obtains that $T=\sum_{\alpha}D^{\alpha}F_{\alpha}$ which completes the first part of the proposition when $1<p\leq \infty$. Note that the case $p=1$ follows from this for any $\tilde{h}>h$.
\end{proof}

\section{Vector-valued spaces of ultradistributions}

Let now $E$ be a complete l.c.s. As we saw above, $\DD'^{(s)}_{L^1}(U)$ and $\DD'^{(s)}_{L^p,h}(U)$, $1\leq p\leq\infty$, are continuously injected in $\DD'^{(s)}(U)$. Following Komatsu \cite{k82}, (see also \cite{ku113}) we define the spaces $\DD'^{(s)}_{L^1}(U;E)$ and $\DD'^{(s)}_{L^p,h}(U;E)$, $1\leq p\leq \infty$, of $E$-valued ultradistributions of type $\DD'^{(s)}_{L^1}(U)$ and $\DD'^{(s)}_{L^p,h}(U)$ respectively, as
\beq
\DD'^{(s)}_{L^1}(U;E)&=&\DD'^{(s)}_{L^1}(U)\varepsilon E=\mathcal{L}_{\epsilon}\left(\left(\DD'^{(s)}_{L^1}(U)\right)'_c, E\right), \mbox{ resp.}\label{1397}\\
\DD'^{(s)}_{L^p,h}(U;E)&=&\DD'^{(s)}_{L^p,h}(U)\varepsilon E=\mathcal{L}_{\epsilon}\left(\left(\DD'^{(s)}_{L^p,h}(U)\right)'_c, E\right)\label{1399}.
\eeq
The subindex $c$ stands for the topology of compact convex circled convergence on the dual of $\DD'^{(s)}_{L^1}(U)$, resp. $\DD'^{(s)}_{L^p,h}(U)$, from the duality
\beqs
\left\langle \DD'^{(s)}_{L^1}(U),\left(\DD'^{(s)}_{L^1}(U)\right)'\right\rangle, \mbox{ resp. } \left\langle \DD'^{(s)}_{L^p,h}(U),\left(\DD'^{(s)}_{L^p,h}(U)\right)'\right\rangle.
\eeqs
If we denote by $\iota$, resp. $\iota_p$, the inclusion $\DD'^{(s)}_{L^1}(U)\rightarrow\DD'^{(s)}(U)$, resp. $\DD'^{(s)}_{L^p,h}(U)\rightarrow\DD'^{(s)}(U)$, then $\DD'^{(s)}_{L^1}(U;E)$, resp. $\DD'^{(s)}_{L^p,h}(U;E)$, is continuously injected into $\DD'^{(s)}(U;E)=\DD'^{(s)}(U)\varepsilon E=\mathcal{L}_b\left(\DD^{(s)}(U),E\right)$ by the mapping $\iota\,\varepsilon\, \mathrm{Id}$, resp. $\iota_p\,\varepsilon\,\mathrm{Id}$ (cf. \cite{k82}). In \cite{SchwartzV} is proved that when both spaces are complete. The same holds for their $\varepsilon$ tensor product. Hence, $\DD'^{(s)}_{L^1}(U;E)$ and $\DD'^{(s)}_{L^p,h}(U;E)$ are complete. Since $\DD'^{(s)}_{L^1}(U)$ and $\DD'^{(s)}_{L^p,h}(U)$ are barrelled (the former is a $(DFS)$-space as the strong dual of a $(FS)$-space, hence barrelled), every bounded subset of $\left(\DD'^{(s)}_{L^1}(U)\right)'_c$ or $\left(\DD'^{(s)}_{L^p,h}(U)\right)'_c$ is equicontinuous (and vice versa). Hence, the $\epsilon$ topology on the right hand sides of (\ref{1397}) and (\ref{1399}) is the same as the topology of bounded convergence. Moreover, since $\dot{\BB}^{(s)}(U)$ is a $(FS)$-space and $\DD'^{(s)}_{L^1}(U)$ is a $(DFS)$-space they are both Montel spaces. Hence $\DD'^{(s)}_{L^1}(U;E)=\mathcal{L}_b\left(\dot{\BB}^{(s)}(U), E\right)$. For $1<p<\infty$, $\DD'^{(s)}_{L^p,h}(U;E)=\mathcal{L}_b\left(\DD^{(s)}_{L^q,h}(U)_c, E\right)$, since $\DD^{(s)}_{L^q,h}(U)$ are reflexive, where $\DD^{(s)}_{L^q,h}(U)_c$ is the space $\DD^{(s)}_{L^q,h}(U)$ equipped with topology of compact convex circled convergence from the duality $\left\langle \DD^{(s)}_{L^q,h}(U),\DD'^{(s)}_{L^p,h}(U)\right\rangle$. Since $\dot{\BB}^{(s)}(U)$ is a nuclear $(FS)$-space (by Proposition \ref{60}) $\DD'^{(s)}_{L^1}(U)$ is a nuclear $(DFS)$-space and hence it satisfies the weak approximation property by Corollary 2 pg.110 of \cite{Schaefer} (for the definition of the weak approximation property see \cite{SchwartzV}). Hence Proposition 1.4 of \cite{k82} implies $\DD'^{(s)}_{L^1}(U;E)=\DD'^{(s)}_{L^1}(U)\varepsilon E\cong \DD'^{(s)}_{L^1}(U)\hat{\otimes} E$ where the $\pi$ and the $\epsilon$ topologies coincide on $\DD'^{(s)}_{L^1}(U)\hat{\otimes} E$ since $\DD'^{(s)}_{L^1}(U)$ is nuclear. Later we will need the following kernel theorem.

\begin{thm}\label{1310}
Let $U_1$ and $U_2$ be bounded open sets in $\RR^{d_1}_x$ and $\RR^{d_2}_y$ respectively. Then we have the following canonical isomorphisms of l.c.s.
\begin{itemize}
\item[$i)$] $\dot{\BB}^{(s)}(U_1)\hat{\otimes}\dot{\BB}^{(s)}(U_2)\cong\dot{\BB}^{(s)}(U_1\times U_2)$.
\item[$ii)$] $\DD'^{(s)}_{L^1}(U_1)\hat{\otimes}\DD'^{(s)}_{L^1}(U_2)\cong\DD'^{(s)}_{L^1}(U_1\times U_2)\cong \DD'^{(s)}_{L^1}(U_1)\varepsilon \DD'^{(s)}_{L^1}(U_2)\newline\cong\mathcal{L}_b\left(\dot{\BB}^{(s)}(U_1), \DD'^{(s)}_{L^1}(U_2)\right)\cong\DD'^{(s)}_{L^1}\left(U_1;\DD'^{(s)}_{L^1}(U_2)\right)\cong \DD'^{(s)}_{L^1}\left(U_2;\DD'^{(s)}_{L^1}(U_1)\right)$.
\end{itemize}
\end{thm}

\begin{proof} First we prove $i)$. Since $\dot{\BB}^{(s)}(U_1)$ and $\dot{\BB}^{(s)}(U_2)$ are nuclear (Proposition \ref{60}) the $\pi$ and the $\epsilon$ topologies coincide on $\dot{\BB}^{(s)}(U_1)\otimes\dot{\BB}^{(s)}(U_2)$. Moreover, one easily verifies that $\dot{\BB}^{(s)}(U_1)\otimes\dot{\BB}^{(s)}(U_2)$ can be regarded as a subspace of $\dot{\BB}^{(s)}(U_1\times U_2)$ by identifying $\varphi\otimes\psi$ with $\varphi(x)\psi(y)$. Since $\DD^{(s)}(U_1\times U_2)$ is continuously and densely injected in $\dot{\BB}^{(s)}(U_1\times U_2)$ and $\DD^{(s)}(U_1)\otimes\DD^{(s)}(U_2)$ is a dense subspace of $\DD^{(s)}(U_1\times U_2)$ (see Theorem 2.1 of \cite{Komatsu2}) we obtain that $\DD^{(s)}(U_1)\otimes\DD^{(s)}(U_2)$ and hence $\dot{\BB}^{(s)}(U_1)\otimes\dot{\BB}^{(s)}(U_2)$ is a dense subspace of $\dot{\BB}^{(s)}(U_1\times U_2)$. Observe that the bilinear mapping $(\varphi,\psi)\mapsto \varphi(x)\psi(y)$, $\dot{\BB}^{(s)}(U_1)\times\dot{\BB}^{(s)}(U_2)\rightarrow \dot{\BB}^{(s)}(U_1\times U_2)$ is continuous (it is separately continuous and hence continuous since all spaces under consideration are $(F)$-spaces). We obtain that the $\pi$ topology on $\dot{\BB}^{(s)}(U_1)\otimes\dot{\BB}^{(s)}(U_2)$ is stronger than the induced one by $\dot{\BB}^{(s)}(U_1\times U_2)$. Hence, to obtain $\dot{\BB}^{(s)}(U_1)\hat{\otimes}\dot{\BB}^{(s)}(U_2)\cong\dot{\BB}^{(s)}(U_1\times U_2)$, it is enough to prove that the $\epsilon$ topology on $\dot{\BB}^{(s)}(U_1)\otimes\dot{\BB}^{(s)}(U_2)$ is weaker than the induced one by $\dot{\BB}^{(s)}(U_1\times U_2)$. Let $A'$ and $B'$ be equicontinuous subsets of $\DD'^{(s)}_{L^1}(U_1)$ and $\DD'^{(s)}_{L^1}(U_2)$ respectively. Hence, there exist $h,C>0$ such that
\beqs
\sup_{T\in A'}|\langle T,\varphi\rangle|\leq C\sup_{x,\alpha}\frac{h^{|\alpha|}\left|D^{\alpha}\varphi(x)\right|}{\alpha!^s} \mbox{ and } \sup_{S\in B'}|\langle S,\psi\rangle|\leq C\sup_{y,\beta}\frac{h^{|\beta|}\left|D^{\beta}\psi(y)\right|}{\beta!^s}
\eeqs
Then for $\chi\in \dot{\BB}^{(s)}(U_1)\otimes\dot{\BB}^{(s)}(U_2)$, $T\in A'$ and $S\in B'$, we have\\
\\
$|\langle T(x)\otimes S(y),\chi(x,y)\rangle|$
\beqs
&=&|\langle T(x),\langle S(y),\chi(x,y)\rangle\rangle|\leq C\sup_{x,\alpha}\frac{h^{|\alpha|}\left|\langle S(y),D^{\alpha}_x\chi(x,y)\rangle\right|}{\alpha!^s}\\
&\leq& C^2\sup_{x,y,\alpha,\beta}\frac{h^{|\alpha|+|\beta|}\left|D^{\alpha}_xD^{\beta}_y \chi(x,y)\right|}{\alpha!^s \beta!^s}\leq C^2\sup_{x,y,\alpha,\beta}\frac{(2^sh)^{|\alpha|+|\beta|}\left|D^{\alpha}_xD^{\beta}_y \chi(x,y)\right|}{(\alpha+\beta)!^s}.
\eeqs
Hence, we obtain that the $\epsilon$ topology is weaker than the topology induced by $\dot{\BB}^{(s)}(U_1\times U_2)$.\\
\indent $ii)$  Since $\dot{\BB}^{(s)}(U_1)$ and $\dot{\BB}^{(s)}(U_2)$ are nuclear $(FS)$-spaces (by Proposition \ref{60}), $\DD'^{(s)}_{L^1}(U_1)$ and $\DD'^{(s)}_{L^1}(U_2)$ are nuclear $(DFS)$-spaces. Hence the $\pi$ and the $\epsilon$ topologies on the tensor product $\DD'^{(s)}_{L^1}(U_1)\otimes\DD'^{(s)}_{L^1}(U_2)$ coincide and by $i)$ (using the fact that $\DD'^{(s)}_{L^1}(U_1)$ and $\DD'^{(s)}_{L^1}(U_2)$ are nuclear $(DFS)$-spaces) we have $\DD'^{(s)}_{L^1}(U_1\times U_2)\cong \left(\dot{\BB}^{(s)}(U_1)\hat{\otimes}\dot{\BB}^{(s)}(U_2)\right)'\cong \DD'^{(s)}_{L^1}(U_1)\hat{\otimes}\DD'^{(s)}_{L^1}(U_2)$. Other isomorphisms in the assertion on $U$ follow by the discussion before the theorem.
\end{proof}

\subsection{Banach-valued ultradistributions}

Let now $E$ be a $(B)$-space and denote by $L^p(U;E)$, $1\leq
p\leq \infty$, the Bochner $L^p$ space. If $\varphi\in
\mathcal{C}_{L^{\infty}}(U)$ (the space of bounded continuous
functions on $U$) and $\mathbf{F}\in L^1(U;E)$ then one easily
verifies that $\varphi \mathbf{F}\in L^1(U;E)$. We will need the
following lemma.

\begin{lem}\label{150}(variant of du Bois-Reymond lemma for Bochner integrable functions)
Let $\mathbf{F}\in L^1(U;E)$ is such that $\ds\int_U \mathbf{F}(x)\varphi(x)dx=0$ for all $\varphi\in\DD^{(s)}(U)$. Then $\mathbf{F}(x)=0$ a.e.
\end{lem}

\begin{proof} Observe first that for each $e'\in E'$ and $\varphi\in\DD^{(s)}(U)$, we have
\beqs
\int_U e'\circ\mathbf{F}(x) \varphi(x)dx=e'\left(\int_U \mathbf{F}(x) \varphi(x)dx\right)=0.
\eeqs
Since $\DD^{(s)}(U)$ is dense in $\DD(U)$, by the du Bois-Reymond lemma it follows that $e'\circ\mathbf{F}=0$ a.e. for each $e'\in E'$. Since $\mathbf{F}$ is strongly measurable $\mathbf{F}(U)$ is separable subset of $E$. Let $D$ be a countable dense subset of $\mathbf{F}(U)$. Denote by $L$ the set of all finite linear combinations of the elements of $D$ with scalars from $\QQ+i\QQ$. Then $L$ is countable. Denote by $\tilde{E}$ the closure of $L$ in $E$. Then $\tilde{E}$ is a separable $(B)$-space and $\mathbf{F}(U)\subseteq \tilde{E}$. Thus $\tilde{E}'_{\sigma}$ is separable (by Theorem 1.7 of Chapter 4 of \cite{Schaefer}; $\sigma$ stands for the weak* topology). Let $\tilde{V}=\{\tilde{e}'_1,\tilde{e}'_2,\tilde{e}'_3,...\}$ be a countable dense subset of $\tilde{E}'_{\sigma}$. Extend each $\tilde{e}'_j$, $j\in\ZZ_+$, by the Hahn-Banach theorem to a continuous functional of $E$ and denote this extension by $e'_j$, $j\in\ZZ_+$. Arguments given above imply that $e'_j\circ \mathbf{F}=0$ a.e. for each $j\in\ZZ_+$ and in fact $\tilde{e}'_j\circ \mathbf{F}=0$ a.e., $j\in\ZZ_+$, since $e'_j$ is extension of $\tilde{e}'_j$ and $\mathbf{F}(U)\subseteq \tilde{E}$. Hence $P_j=\{x\in U|\, \tilde{e}'_j\circ \mathbf{F}(x)\neq 0\}$ is a set of measure $0$, for each $j\in\ZZ_+$ and so is $P=\bigcup_j P_j$. We will prove that $\mathbf{F}(x)=0$ for every $x\in U\backslash P$. Assume that there exists $x_0\in U\backslash P$ such that $\mathbf{F}(x_0)\neq 0$. Then there exists $\tilde{e}'\in \tilde{E}'$ such that $\tilde{e}'\circ\mathbf{F}(x_0)\neq 0$ i.e. $\left|\tilde{e}'\circ\mathbf{F}(x_0)\right|=c>0$. Then there exists $\tilde{e}'_k\in \tilde{V}$ such that $\left|\tilde{e}'\circ\mathbf{F}(x_0)-\tilde{e}'_k\circ\mathbf{F}(x_0)\right|\leq c/2$. Since $\tilde{e}'_k\circ\mathbf{F}(x_0)=0$, by the definition of $P$, we have
\beqs
c=\left|\tilde{e}'\circ\mathbf{F}(x_0)\right|\leq \left|\tilde{e}'\circ\mathbf{F}(x_0)-\tilde{e}'_k\circ\mathbf{F}(x_0)\right|+\left|\tilde{e}'_k\circ\mathbf{F}(x_0)\right|\leq c/2,
\eeqs
which is a contradiction. Hence $\mathbf{F}(x)=0$ for all $x\in U\backslash P$ and the proof is complete.
\end{proof}

Denote by $\delta_x$ the delta ultradistribution concentrated at $x$. For $\alpha\in\NN^d$ and $x\in U$ one easily verifies that $D^{\alpha}\delta_x \in \DD'^{(s)}_{L^1,h}(U)$ for any $h>0$ and hence, by Proposition \ref{50}, $D^{\alpha}\delta_x \in \DD'^{(s)}_{L^p,h}(U)$ for any $h>0$ and $1\leq p\leq \infty$. For the next proposition we need the following result.

\begin{lem}\label{310}
Let $h>0$, $\alpha\in\NN^d$ and $1\leq p\leq \infty$. The set $G_{\alpha}=\{D^{\alpha}\delta_x|\, x\in U\}\subseteq \DD'^{(s)}_{L^p,h}(U)$ is precompact in $\DD'^{(s)}_{L^p,h}(U)$.
\end{lem}

\begin{proof} Let $0<h_1<h/2^s$. By Proposition \ref{50} we have the continuous inclusion $\DD^{(s)}_{L^q,h}(U)\rightarrow\DD^{(s)}_{L^{\infty},h/2^s}(U)$. Proposition \ref{60} implies that the inclusion $\DD^{(s)}_{L^{\infty},h/2^s}(U)\rightarrow\DD^{(s)}_{L^{\infty},h_1}(U)$ is compact. Hence we have the compact dense inclusion $\DD^{(s)}_{L^q,h}(U)\rightarrow \DD^{(s)}_{L^{\infty},h_1}(U)$ (the denseness follows from the fact that $\DD^{(s)}(U)\subseteq \DD^{(s)}_{L^q,h}(U)$ is dense in $\DD^{(s)}_{L^{\infty},h_1}(U)$). So, the dual mapping $\DD'^{(s)}_{L^1,h_1}(U)\rightarrow \DD'^{(s)}_{L^p,h}(U)$ is compact inclusion. Observe that, for $\varphi\in \DD^{(s)}_{L^{\infty},h_1}(U)$, $\ds\left|\langle D^{\alpha}\delta_x,\varphi\rangle\right|\leq \frac{\alpha!^s}{h_1^{|\alpha|}}\left\|D^{\alpha}\varphi\right\|_{\DD^{(s)}_{L^{\infty},h_1}(U)}$, $\forall x\in U$. Hence $G_{\alpha}$ is bounded in the $(B)$-space $\DD'^{(s)}_{L^1,h_1}(U)$, thus precompact in $\DD'^{(s)}_{L^p,h}(U)$.
\end{proof}

\begin{prop}\label{54}
Each $\mathbf{F}\in L^p(U;E)$ can be regarded as an $E$-valued ultradistribution by $\ds\overline{\mathbf{F}}(\varphi)=\int_U \mathbf{F}(x)\varphi(x)dx$. In this way $L^p(U;E)$ is continuously injected into $\DD'^{(s)}_{L^1}(U;E)$ for $1\leq p\leq \infty$ and in $\DD'^{(s)}_{L^p,h}(U;E)$ for $1<p <\infty$.
\end{prop}

\begin{proof} Let $\mathbf{F}\in L^p(U;E)$. First we will prove that $L^p(U;E)$ is continuously injected into $\DD'^{(s)}_{L^1}(U;E)$. If $\varphi\in\dot{\BB}^{(s)}(U)$ then
\beq\label{70}
\left\|\int_U \mathbf{F}(x)\varphi(x)dx\right\|_E\leq \int_U \|\mathbf{F}(x)\|_E|\varphi(x)|dx\leq \|\mathbf{F}\|_{L^p(U;E)}\|\varphi\|_{L^q(U)}.
\eeq
Since $U$ is bounded, $\|\varphi\|_{L^q(U)}\leq |U|^{1/q}\|\varphi\|_{L^{\infty}(U)}$. Hence $\overline{\mathbf{F}}\in \mathcal{L}_b\left(\dot{\BB}^{(s)}(U),E\right)=\DD'^{(s)}_{L^1}(U;E)$ and the mapping $\mathbf{F}\mapsto \overline{\mathbf{F}}$ is continuous from $L^p(U;E)$ into $\DD'^{(s)}_{L^1}(U;E)$. To prove that it is injective let $\overline{\mathbf{F}}=0$ i.e. $\ds \int_U \mathbf{F}(x)\varphi(x)dx=0$ for all $\varphi\in\dot{\BB}^{(s)}(U)$. Since $U$ is bounded $L^p(U;E)\subseteq L^1(U;E)$. Now, Lemma \ref{150} implies that $\mathbf{F}=0$.\\
\indent Next, we prove that $L^p(U;E)$ is continuously injected into $\DD'^{(s)}_{L^p,h}(U;E)$ for $1<p<\infty$. Consider the set $G=\{\delta_x|\, x\in U\}\subseteq \DD'^{(s)}_{L^p,h}(U)$. It is precompact in $\DD'^{(s)}_{L^p,h}(U)$ by Lemma \ref{310}. Fix $\mathbf{F}\in L^p(U;E)$ and note that (\ref{70}) still holds when $\varphi\in\DD^{(s)}_{L^q,h}(U)$. Let $V=\{e\in E|\, \|e\|_E\leq \varepsilon\}$ be a neighborhood of zero in $E$ and $\ds \tilde{G}= \frac{\|\mathbf{F}\|_{L^p(U;E)} |U|^{1/q}}{\varepsilon}G$. Since $G$ is precompact so is $\tilde{G}$. But then, for $\varphi\in \tilde{G}^{\circ}$,
\beqs
\|\mathbf{F}\|_{L^p(U;E)}\|\varphi\|_{L^q(U)}\leq |U|^{1/q}\|\mathbf{F}\|_{L^p(U;E)}\sup_{x\in U}|\langle \delta_x, \varphi\rangle|\leq \varepsilon.
\eeqs
Hence $\overline{\mathbf{F}}(\varphi)\in V$ for all $\varphi\in \tilde{G}^{\circ}$. We obtain that $\overline{\mathbf{F}}\in \mathcal{L}\left(\DD^{(s)}_{L^q,h}(U)_c, E\right)$ since the topology of precompact convergence on $\DD^{(s)}_{L^q,h}(U)$ coincides with the topology of compact convex circled convergence ($\DD'^{(s)}_{L^p,h}(U)$ is a $(B)$-space). The continuity of the mapping $\mathbf{F}\mapsto\overline{\mathbf{F}}$ follows from (\ref{70}) since the bounded sets of $\DD^{(s)}_{L^q,h}(U)$ are the same for the initial topology and the topology of compact convex circled convergence. The proof of the injectivity is the same as above.
\end{proof}

By Proposition \ref{54}, from now on we will use the same notation for $\mathbf{F}\in L^p(U;E)$ and its image in $\DD'^{(s)}_{L^1}(U;E)$, resp. $\DD'^{(s)}_{L^p,h}(U;E)$ for $1<p<\infty$.\\
\indent For $\alpha\in\NN^d$ and $\mathbf{F}\in L^p(U;E)$, $1<p<\infty$, define $D^{\alpha}\mathbf{F}\in \DD'^{(s)}_{L^p,h}(U;E)$ by
\beqs
D^{\alpha}\mathbf{F}(\varphi)=\int_U \mathbf{F}(x) (-D)^{\alpha}\varphi(x)dx,\, \varphi\in \DD^{(s)}_{L^q,h}(U).
\eeqs
As in Proposition \ref{54}, one can prove that this is well defined element of $\DD'^{(s)}_{L^p,h}(U;E)$. One only has to use the set $G_{\alpha}$ from Lemma \ref{310} instead $G=\{\delta_x|\, x\in U\}$. Observe that $D^{\alpha}\mathbf{F}$ coincides with the ultradistributional derivative of $\mathbf{F}$ when we regard $\mathbf{F}$ as an element of $\DD'^{(s)}_{L^1}(U;E)$ or $\DD'^{(s)}(U;E)$.

\begin{thm}\label{350}
Let $1<p<\infty$ and $\mathbf{F}_{\alpha}\in L^p(U;E)$, $\alpha\in\NN^d$, are such that, for some fixed $h>0$, $\left(\ds \sum_{\alpha}\frac{\alpha!^{ps}}{h^{|\alpha|p}}\|\mathbf{F}_{\alpha}\|_{L^p}^p\right)^{1/p}<\infty$. Then the partial sums $\ds\sum_{|\alpha|=0}^n D^{\alpha} \mathbf{F}_{\alpha}$ converge absolutely in $\DD'^{(s)}_{L^p}(U;E)$ and $\DD'^{(s)}_{L^p,h}(U;E)$.\\
\indent The partial sums converge absolutely in $\DD'^{(s)}_{L^1}(U;E)$ also in the cases $p=1$ and $p=\infty$.
\end{thm}

\begin{proof} Let $1<p<\infty$. To prove that the partial sums converge absolutely in $\DD'^{(s)}_{L^p,h}(U;E)=\mathcal{L}_b\left(\DD^{(s)}_{L^q,h}(U)_c, E\right)$ let $B$ be a bounded subset of $\DD^{(s)}_{L^q,h}(U)_c$. Since the bounded sets of $\DD^{(s)}_{L^q,h}(U)$ are the same for the initial topology and the topology of compact convex circled convergence we may assume that $B$ is the closed unit ball in $\DD^{(s)}_{L^q,h}(U)$. We obtain\\
\\
$\ds \sum_{|\alpha|=0}^n \sup_{\varphi\in B}\left\|\int_U \mathbf{F}_{\alpha}(x)(-D)^{\alpha}\varphi(x)dx\right\|_E$
\beqs
&\leq& \sup_{\varphi\in B}\sum_{|\alpha|=0}^{\infty} \int_U \|\mathbf{F}_{\alpha}(x)\|_E |D^{\alpha}\varphi(x)|dx\leq \sup_{\varphi\in B} \sum_{|\alpha|=0}^{\infty} \|\mathbf{F}_{\alpha}\|_{L^p(U;E)} \|D^{\alpha}\varphi\|_{L^q(U)}\\
&\leq& \left(\sum_{|\alpha|=0}^{\infty} \frac{\alpha!^{ps}}{h^{|\alpha|p}}\|\mathbf{F}_{\alpha}\|_{L^p(U;E)}^p\right)^{1/p} \cdot \sup_{\varphi\in B}\left(\sum_{|\alpha|=0}^{\infty}\frac{h^{|\alpha|q}}{\alpha!^{qs}}\|D^{\alpha}\varphi\|_{L^q(U)}^q\right)^{1/q},
\eeqs
for any $n\in\ZZ_+$. Since $\DD'^{(s)}_{L^p,h}(U;E)$ is complete it follows that the partial sums converge absolutely in $\DD'^{(s)}_{L^p,h}(U;E)$ to an element of $\DD'^{(s)}_{L^p,h}(U;E)$. The proof for $\DD'^{(s)}_{L^1}(U;E)$ is similar.
\end{proof}

Observe that each $\mathbf{F}\in\mathcal{C}(\overline{U};E)$ is in $L^p(U;E)$ for any $1\leq p\leq \infty$. To see this, note that $\mathbf{F}$ is separately valued since it is continuous and $\overline{U}$ is a subset of $\RR^d$. Moreover it is easy to see that it is weakly measurable. Hence Pettis' theorem implies that $\mathbf{F}$ is strongly measurable. Now the claim follows since $U$ is bounded $\|\mathbf{F}(\cdot)\|_E$ is in $L^p(U)$, for any $1\leq p\leq \infty$.

\begin{thm}\label{990}
Let $\mathbf{f}\in \DD'^{(s)}_{L^1}(U;E)$ and $1\leq p\leq \infty$. Then there exists $h>0$ and $\mathbf{F}_{\alpha}\in \mathcal{C}(\overline{U};E)$, $\alpha\in\NN^d$, such that
\beq\label{370}
\left(\sum_{\alpha}\frac{\alpha!^{ps}}{h^{|\alpha|p}}\|\mathbf{F}_{\alpha}\|_{L^p(U;E)}^p\right)^{1/p}<\infty
\eeq
and $\ds\mathbf{f}=\sum_{|\alpha|=0}^{\infty} D^{\alpha} \mathbf{F}_{\alpha}$, where the series converges absolutely in $\DD'^{(s)}_{L^1}(U;E)$.\\
\indent Conversely, let $\mathbf{F}_{\alpha}\in L^p(U;E)$, $\alpha\in\NN^d$, be such that (\ref{370}) holds. Then there exists $\mathbf{f}\in \DD'^{(s)}_{L^1}(U;E)$ such that $\ds\mathbf{f}=\sum_{|\alpha|=0}^{\infty} D^{\alpha} \mathbf{F}_{\alpha}$ and the series converges absolutely in $\DD'^{(s)}_{L^1}(U;E)$.
\end{thm}

\begin{proof} First, note that the second part of the theorem follows by Theorem \ref{350}. To prove the first part, let $\mathbf{f}\in\DD'^{(s)}_{L^1}(U;E)=\mathcal{L}_b\left(\dot{\BB}^{(s)}(U), E\right)$. Since $\dot{\BB}^{(s)}(U)$ is nuclear (by Proposition \ref{60}) and $E$ is a $(B)$-space $\mathbf{f}$ is nuclear. Hence there exists a sequence $e_j$, $j\in\NN$, in the closed unit ball of $E$, an equicontinuous sequence $f_j$, $j\in\NN$, of $\DD'^{(s)}_{L^1}(U)$ and a complex sequence $\lambda_j$, $j\in\NN$, such that $\sum_j |\lambda_j|<\infty$, such that
\beqs
\mathbf{f}(\varphi)=\sum_{j=0}^{\infty}\lambda_j \langle f_j,\varphi\rangle e_j.
\eeqs
Since $\{f_j|j\in\NN\}$ is equicontinuous subset of $\DD'^{(s)}_{L^1}(U)$, it is bounded and by Proposition \ref{330}, there exist $h,C>0$ and $F_{j,\alpha}\in \mathcal{C}(\overline{U})$ such that
\beqs
f_j=\sum_{|\alpha|=0}^{\infty} D^{\alpha} F_{j,\alpha} \mbox{ and } \sup_j \left(\sum_{\alpha\in\NN^d} \frac{\alpha!^{ps}}{h^{|\alpha|p}}\|F_{j,\alpha}\|_{L^{\infty}(U)}^p\right)^{1/p}\leq C.
\eeqs
Define $\mathbf{F}_{\alpha}(x)=\sum_j \lambda_j F_{j,\alpha}(x)e_j$. To prove that $\mathbf{F}_{\alpha}\in \mathcal{C}(\overline{U};E)$, observe that for each $j\in\NN$, $\lambda_j F_{j,\alpha}(x)e_j\in \mathcal{C}(\overline{U};E)$ and the series $\sum_j \lambda_j F_{j,\alpha}(x)e_j$ converges absolutely in the $(B)$-space $\mathcal{C}(\overline{U};E)$. Hence $\mathbf{F}_{\alpha}\in\mathcal{C}(\overline{U};E)$. Moreover
\beqs
\frac{\alpha!^s}{h^{|\alpha|}}\|\mathbf{F}_{\alpha}(x)\|_E\leq \sum_{j=0}^{\infty}|\lambda_j|\frac{\alpha!^s}{h^{|\alpha|}}\left\|F_{j,\alpha}\right\|_{L^{\infty}(U)}\leq C\sum_{j=0}^{\infty}|\lambda_j|,\quad \mbox{for all}\,\, x\in \overline{U}.
\eeqs
We obtain $\ds \sup_{\alpha}\frac{\alpha!^s}{h^{|\alpha|}}\|\mathbf{F}_{\alpha}\|_{\mathcal{C}(\overline{U};E)}<\infty$. Since $U$ is bounded, (\ref{370}) holds for any $h_1>h$. One easily verifies that the series $\sum_{j,\alpha}\lambda_j \langle D^{\alpha} F_{j,\alpha},\varphi\rangle e_j$ converges absolutely in $E$ for each fixed $\varphi\in \dot{\BB}^{(s)}(U)$. Hence $\ds\mathbf{f}(\varphi)=\sum_{|\alpha|=0}^{\infty} D^{\alpha}\mathbf{F}_{\alpha}(\varphi)$, for each fixed $\varphi\in \dot{\BB}^{(s)}(U)$. By Theorem \ref{350}, $\ds\sum_{|\alpha|=0}^{\infty} D^{\alpha}\mathbf{F}_{\alpha}$ converges absolutely in $\DD'^{(s)}_{L^1}(U;E)$, hence $\ds\mathbf{f}=\sum_{|\alpha|=0}^{\infty} D^{\alpha}\mathbf{F}_{\alpha}$.
\end{proof}

\section{On the  Cauchy problem in  $\tilde{{\mathcal D}}'^s_{L^p,h}(0,T;E)$}\label{910}

In this section $E$ is the $(B)$-space with the norm $\|\cdot\|$, and $D(A)$ is the domain of a closed linear operator $A$, endowed with the graph norm $\|u\|_{D(A)}=\|u\|+\|Au\|$. We use standard notation for the symbols $R(\lambda:A)$, $\rho(A)$. The results obtained in previous sections will often be applied in the one dimensional case (i.e. $d=1$) when a bounded open set $U$ is equal to the interval $(0,T)$. In this case we will use the more descriptive notations $L^p(0,T;E)$, $\DD^s_{L^p,h}(0,T)$, $\DD^{(s)}_{L^p,h}(0,T)$, $\dot{\BB}^{(s)}(0,T)$, $\DD'^{(s)}_{L^p,h}(0,T)$, $\DD'^{(s)}_{L^1}(0,T)$, $\DD'^{(s)}_{L^p,h}(0,T;E)$ and $\DD'^{(s)}_{L^1}(0,T;E)$ for the spaces $L^p(U;E)$, $\DD^{s}_{L^p,h}(U)$, $\DD^{(s)}_{L^p,h}(U)$, $\dot{\BB}^{(s)}(U)$, $\DD'^{(s)}_{L^p,h}(U)$, $\DD'^{(s)}_{L^1}(U)$, $\DD'^{(s)}_{L^p,h}(U;E)$ and $\DD'^{(s)}_{L^1}(U;E)$, respectively. Note that by Sobolev imbedding theorem, every derivative of $\varphi\in \DD^s_{L^p,h}(0,T)$ can be extended to uniformly continuous function on $[0,T]$. As in \cite{d81}, we define the $E$-valued Sobolev space $W^{1,p}(0,T;E)$ as the space of all $\mathbf{F}:[0,T]\rightarrow E$, such that $\ds \mathbf{F}(t)=F_0+\int_0^t \mathbf{F}'(s)ds$, $t\in[0,T]$, for some $F_0\in E$ and $\mathbf{F}'(t)\in L^p(0,T;E)$, with the norm $\|\mathbf{F}\|_{W^{1,p}(0,T:E)}=\|\mathbf{F}\|_{L^p(0,T;E)}+\|\mathbf{F}'\|_{L^p(0,T;E)}$, $1\leq p<\infty$. Observe that if $\mathbf{F}\in W^{1,p}(0,T;E)$ then $\mathbf{F}$ is continuous function with values in $E$ which is a.e. differentiable and its derivative is equal to $\mathbf{F}'$ a.e.\\
\indent Let $1\leq p<\infty$. Define $\tilde{{\mathcal D}}^{'s}_{L^p,h}(0,T;E)$ as a space of all sequences $\mathbf{f}=(\mathbf{F}_{\alpha})_{\alpha}$, $\mathbf{F}_{\alpha}\in L^p(0,T;E)$, $\alpha\in\NN$, such that
\beq\label{510}
\|\mathbf{f}\|_{{\tilde{\mathcal D}}^{'s}_{L^p,h}(0,T;E)}=\left(\sum_{\alpha\in\NN}\frac{{\alpha}!^{ps}}{h^{p\alpha}} \|\mathbf{F}_{\alpha}\|^p_{L^p(0,T;E)}\right)^{1/p}<\infty.
\eeq
One easily verifies that it is a $(B)$-space with the norm (\ref{510}). Each $\mathbf{f}\in\tilde{{\mathcal D}}_{L^p,h}^{'s}(0,T;E)$ generates an element of $\mathcal{L}\left({\mathcal D}_{L^q,h}^s(0,T),E\right)$ by
\beqs
\langle \mathbf{f},\varphi\rangle=\mathbf{f}(\varphi)=\sum\limits_{\alpha\in \NN}(-1)^{\alpha}\int_0^T \mathbf{F}_{\alpha}(t)\varphi^{(\alpha)}(t)\, dt\in E.
\eeqs
Moreover, one easily verifies that the mapping $\mathbf{f}\mapsto \langle \mathbf{f},\cdot \rangle$, $\tilde{{\mathcal D}}_{L^p,h}^{'s}(0,T;E)\rightarrow \mathcal{L}_b\left({\mathcal D}_{L^q,h}^s(0,T),E\right)$ is continuous.

\begin{rem} It is worth to note that this mapping is not injective. To see this let $\psi\in \DD^{(s)}(0,T)$, $\psi\neq 0$. Take nonzero element $e$ of $E$ and define $\mathbf{F}(x)=\psi'(x)e$ and $\mathbf{G}(x)=\psi(x)e$, $x\in(0,T)$. Obviously $\mathbf{F},\mathbf{G}\in L^p(0,T;E)$, for any $1\leq p\leq \infty$. Define $\mathbf{f},\mathbf{g}\in \tilde{{\mathcal D}}_{L^p,h}^{'s}(0,T;E)$ by $\mathbf{f}=(\mathbf{F},0,0,...)$ and $\mathbf{g}=(0,\mathbf{G},0,...)$. Observe that, for $\varphi\in {\mathcal D}_{L^q,h}^s(0,T)$,
\beqs
\langle \mathbf{f},\varphi\rangle=e\int_0^T \psi'(x)\varphi(x)dx=-e\int_0^T\psi(x)\varphi'(x)dx=\langle \mathbf{g},\varphi\rangle.
\eeqs
Hence $\langle \mathbf{f},\cdot\rangle$ and $\langle \mathbf{g}, \cdot\rangle$ are the same element of $\mathcal{L}_b\left({\mathcal D}_{L^q,h}^s(0,T),E\right)$.
\end{rem}


Note that $L^p(0,T;E)$ can be continuously imbedded in $\tilde{{\mathcal D}}^{'s}_{L^p,h}(0,T;E)$ by $\mathbf{F}\mapsto (\mathbf{F},0,0,\ldots)$.\\
\indent Let $1\leq p<\infty$. We define $\tilde{{\mathcal D}}^{'s}_{W^{1,p},h}(0,T;E)$ as the space of all sequences $\mathbf{f}=(\mathbf{F}_{\alpha})_{\alpha}$, where $\mathbf{F}_{\alpha}\in W^{1,p}(0,T;E)$ and
\beqs
\|\mathbf{f}\|_{\tilde{{\mathcal D}}^{'s}_{W^{1,p},h}(0,T;E)}=\left(\sum_{\alpha\in\NN}\frac{{\alpha}!^{ps}}{h^{p\alpha}} \left(\|\mathbf{F}_{\alpha}\|^p_{L^p(0,T;E)}+\|\mathbf{F}'_{\alpha}\|^p_{L^p(0,T;E)}\right)\right)^{1/p}<\infty.
\eeqs
Equipped with the norm $\|\cdot\|_{\tilde{{\mathcal D}}^{'s}_{W^{1,p},h}(0,T;E)}$, it becomes a $(B)$-space.\\ $\tilde{{\mathcal D}}^{'s}_{W^{1,p},h}(0,T;E)$ is continuously injected into $\tilde{{\mathcal D}}^{'s}_{L^p,h}(0,T;E)$. For $\mathbf{f}=(\mathbf{F}_{\alpha})_{\alpha}\in\tilde{{\mathcal D}}^{'s}_{W^{1,p},h}(0,T;E)$, $\mathbf{f}'=\tilde{\mathbf{f}}=(\tilde{\mathbf{F}}_{\alpha})_{\alpha}\in \tilde{{\mathcal D}}^{'s}_{L^p,h}(0,T;E)$, where $\tilde{\mathbf{F}}_{\alpha}=\mathbf{F}'_{\alpha}$ is the classical derivative a.e. in $(0,T)$.\\ Moreover, the mapping $\mathbf{f}\mapsto \mathbf{f}'$, $\tilde{{\mathcal D}}^{'s}_{W^{1,p},h}(0,T;E)\rightarrow \tilde{{\mathcal D}}^{'s}_{L^p,h}(0,T;E)$, is continuous.\\
\indent Our main assumption is that the Hille-Yosida condition
holds for the resolvent of the operator $A$:
\beq\label{1410}
\|(\lambda-\omega)^k R(\lambda:A)^k\|\leq C, \mbox{ for } \lambda>\omega,\, k\in\ZZ_+.
\eeq
From now on we will always denote these constants by $\omega$ and $C$.

\subsection{Various types of solutions}

We need the following technical lemma.

\begin{lem}\label{1295}
Let $1\leq p<\infty$ and $\mathbf{g}=(\mathbf{G}_{\alpha})_{\alpha}\in\tilde{{\mathcal D}}^{'s}_{L^p,h}(0,T;D(A))$. Then for every $\varphi\in\DD^{s}_{L^q}(0,T)$, $\langle \mathbf{g},\varphi\rangle \in D(A)$ and
\beqs
A\sum_{\alpha=0}^{\infty}(-1)^{\alpha}\int_0^T \mathbf{G}_{\alpha}(t)\varphi^{(\alpha)}(t)dt= \sum_{\alpha=0}^{\infty}(-1)^{\alpha}\int_0^T A\mathbf{G}_{\alpha}(t)\varphi^{(\alpha)}(t)dt.
\eeqs
\end{lem}

\begin{proof} First observe that for each $\alpha\in \NN$, $\mathbf{G}_{\alpha}\varphi^{(\alpha)}\in L^1(0,T;D(A))$ and $A\mathbf{G}_{\alpha}\varphi^{(\alpha)}\in L^1(0,T;E)$ since $\mathbf{G}_{\alpha}(t)\in L^p(0,T;D(A))$ and $\varphi\in \DD^{s}_{L^q}(0,T)$. Then
\beq\label{1270}
A\int_0^T \mathbf{G}_{\alpha}(t)\varphi^{(\alpha)}(t)dt=\int_0^T A\mathbf{G}_{\alpha}(t)\varphi^{(\alpha)}(t)dt.
\eeq
Moreover, observe that
\beqs
\sum_{\alpha=0}^{\infty}\left\|\int_0^T \mathbf{G}_{\alpha}(t)\varphi^{(\alpha)}(t)dt\right\|_{D(A)}\leq \left\|(\mathbf{G}_{\alpha})_{\alpha}\right\|_{\tilde{{\mathcal D}}^{'s}_{L^p,h}(0,T;D(A))}\|\varphi\|_{\DD^{s}_{L^q}(0,T)}.
\eeqs
We obtain that $\ds \sum_{\alpha=0}^{\infty}(-1)^{\alpha}\int_0^T \mathbf{G}_{\alpha}(t)\varphi^{(\alpha)}(t)dt$ converges absolutely in $D(A)$, i.e. $\langle \mathbf{g},\varphi\rangle \in D(A)$. Hence
\beqs
A\sum_{\alpha=0}^{\infty}(-1)^{\alpha}\int_0^T \mathbf{G}_{\alpha}(t)\varphi^{(\alpha)}(t)dt= \sum_{\alpha=0}^{\infty}(-1)^{\alpha}A\int_0^T \mathbf{G}_{\alpha}(t)\varphi^{(\alpha)}(t)dt,
\eeqs
which, together with (\ref{1270}), completes the proof of the lemma.
\end{proof}

Let $u_{0,\alpha}\in E$, $\alpha\in\NN$, be such that
\beq\label{1290}
\left(\sum_{\alpha=0}^{\infty}\frac{\alpha!^{ps}}{h^{p\alpha}}\|u_{0,\alpha}\|^p_E\right)^{1/p}<\infty.
\eeq
Then the constant functions $\tilde{\mathbf{U}}_{\alpha}(t)=u_{0,\alpha}$, $t\in[0,T]$, are such that $\tilde{\mathbf{U}}_{\alpha}\in L^p(0,T;E)$ and (\ref{510}) holds. Hence $(\tilde{\mathbf{U}}_{\alpha})_{\alpha}\in \tilde{{\mathcal D}}'^s_{L^p,h}(0,T;E)$. In the sequel, if $u_{0,\alpha}$, $\alpha\in\NN$, are such elements we will denote the corresponding constant functions simply by $u_{0,\alpha}$ and the element $(u_{0,\alpha})_{\alpha}$ of $\tilde{{\mathcal D}}'^s_{L^p,h}(0,T;E)$ that they generate by $u_0$. We also use the notation $\|u_0\|_{\tilde{{\mathcal D}}'^s_{L^p,h}(0,T;E)}$ for the norm of this element of $\tilde{{\mathcal D}}'^s_{L^p,h}(0,T;E)$.\\
\indent We recall from \cite{d81} the definition of two types of
solutions of the Cauchy problem (\ref{ACP}) (here they are
restated to fit in our setting). We also define weak version
of them. Let $A: D(A)\subseteq E\rightarrow E$ be a closed linear
operator in the
$(B)$-space $E$, $\mathbf{f}\in\tilde{{\mathcal D}}^{'s}_{L^p,h}(0,T;E)$ and $u_{0,\alpha}\in E$, $\alpha\in \NN$.\\
\indent  \textbf{1.} We say that
$\mathbf{u}=(\mathbf{U}_{\alpha})_{\alpha}$ is a strict solution, respectively, strict weak solution,
in $\tilde{{\mathcal D}}^{'s}_{L^p,h}(0,T;E)$ of (\ref{ACP}) if
$\mathbf{u}\in \tilde{{\mathcal
D}}^{'s}_{W^{1,p},h}(0,T;E)\cap\tilde{{\mathcal
D}}^{'s}_{L^p,h}(0,T;D(A))$ and \beqs
\mathbf{U}'_{\alpha}(t)=A\mathbf{U}_{\alpha}(t)+\mathbf{F}_{\alpha}(t),\,
t\in[0,T]\, a.e. \mbox{ and }
\mathbf{U}_{\alpha}(0)=u_{0,\alpha},\, \forall \alpha\in\NN, \eeqs
respectively,
for each $\varphi\in
\DD^{s}_{L^q,h}(0,T)$ it satisfies \beq\label{ACP1} \langle
\mathbf{u}'(t),\varphi(t)\rangle=A\langle
\mathbf{u}(t),\varphi(t)\rangle+\langle
\mathbf{f}(t),\varphi(t)\rangle \mbox{ and }
\mathbf{U}_{\alpha}(0)=u_{0,\alpha},\, \forall \alpha\in\NN. \eeq

We know by Lemma \ref{1295} that
$\langle \mathbf{u}(t),\varphi(t)\rangle\in D(A)$ for each
$\varphi\in\DD^{s}_{L^q,h}(0,T).$
 Also, note that in both cases (of  strict or of strict weak solution
of (\ref{ACP})) we have \beqs \|u_{0,\alpha}\|^p_E\leq
2^pT^{-1}\|\mathbf{U}_{\alpha}\|_{L^p(0,T;E)}+2^pT^{p/q}\|\mathbf{U}'_{\alpha}\|_{L^p(0,T;E)}.
\eeqs
Hence $u_0=(u_{0,\alpha})_{\alpha}$ satisfies (\ref{1290}).\\
\indent \textbf{2.} We say that $\mathbf{u}\in\tilde{{\mathcal
D}}^{'s}_{L^p,h}(0,T;E)$ is an $F$-solution, respectively, $F$-weak solution  in $\tilde{{\mathcal
D}}^{'s}_{L^p,h}(0,T;E)$ of (\ref{ACP}), if for every
$k\in{\mathbb N}$ there is
$\mathbf{u}_k=(\mathbf{U}_{k,\alpha})_{\alpha}\in\tilde{{\mathcal
D}}^{'s}_{W^{1,p},h}(0,T;E)\cap\tilde{{\mathcal
D}}^{'s}_{L^p,h}(0,T;D(A))$ such that from \beqs
\mathbf{U}_{k,\alpha}'(t)=A\mathbf{U}_{k,\alpha}(t)+\mathbf{F}_{k,\alpha}(t),\,
t\in[0,T]\, a.e. \mbox{ and }
\mathbf{U}_{k,\alpha}(0)=u_{0,k,\alpha} \eeqs we have \beqs
\lim\limits_{k\rightarrow\infty}\Big(\|\mathbf{u}_k-\mathbf{u}\|_{\tilde{{\mathcal
D}}^{'s}_{L^p,h}(0,T;E)}+
\|\mathbf{f}_k-\mathbf{f}\|_{\tilde{{\mathcal
D}}^{'s}_{L^p,h}(0,T;E)}+\eeqs  \beqs+\|u_{0,k}-u_0\|_{\tilde{{\mathcal
D}}'^s_{L^p,h}(0,T;E)}\Big)=0, \eeqs
respectively,   from \beqs
\langle \mathbf{u}_k'(t),\varphi(t)\rangle&=&A\langle \mathbf{u}_k(t),\varphi(t)\rangle+\langle \mathbf{f}_k(t),\varphi(t)\rangle,\, \forall \varphi\in \DD^s_{L^q,h}(0,T)\\
&{}&\mbox{and } \mathbf{U}_{k,\alpha}(0)=u_{0,k,\alpha},\, \forall k,\alpha\in\NN
\eeqs
we have that for every $\varphi\in{\mathcal D}^s_{L^q,h}(0,T)$,
\beq\label{F-solution2}
\lim\limits_{k\rightarrow\infty}\left(\|\langle \mathbf{u}_k-\mathbf{u},\varphi\rangle\|_E+
\|\langle \mathbf{f}_k-\mathbf{f},\varphi\rangle\|_E+\|\langle u_{0,k}-u_0,\varphi\rangle\|_E\right)=0.
\eeq
From the above definitions it is clear that a strict, resp. a strict weak solution, in
$\tilde{{\mathcal D}}^{'s}_{L^p,h}(0,T;E)$ is an $F$-solution, resp. $F$-weak solution in $\tilde{{\mathcal D}}^{'s}_{L^p,h}(0,T;E)$.

\begin{rem} If a strict weak solution of (\ref{ACP}) in $\tilde{{\mathcal D}}^{'s}_{L^p,h}(0,T;E)$ exists then it is not unique. To see this let $\psi\in \DD^{(s)}(0,T)$ and $e\in D(A)$ such that $\psi\neq 0$ and $e\neq 0$. Define $\mathbf{v}=(\mathbf{V}_{\alpha})_{\alpha}\in\tilde{{\mathcal D}}^{'s}_{L^p,h}(0,T;E)$ by $\mathbf{V}_0(t)=\psi'(t)e$, $\mathbf{V}_1(t)=-\psi(t)e$ and $\mathbf{V}_{\alpha}(t)=0$, for $\alpha\geq 2$, $\alpha\in \NN$. Obviously $\mathbf{v}\in\tilde{{\mathcal D}}^{'s}_{W^{1,p},h}(0,T;E)\cap\tilde{{\mathcal D}}^{'s}_{L^p,h}(0,T;D(A))$ and $\mathbf{V}_{\alpha}(0)=0$, $\forall \alpha\in\NN$. Moreover, it is easy to verify that the operators $\langle \mathbf{v},\cdot\rangle,\langle \mathbf{v}',\cdot\rangle\in \mathcal{L}\left({\mathcal D}_{L^q,h}^s(0,T),E\right)$ are in fact the zero operator. Hence, if $\mathbf{u}$ is a strict weak solution of (\ref{ACP}) in $\tilde{{\mathcal D}}^{'s}_{L^p,h}(0,T;E)$ then so is $\mathbf{u}+\mathbf{v}$.\\
\indent One can use the same construction to prove that the $F$-weak solution in $\tilde{{\mathcal D}}^{'s}_{L^p,h}(0,T;E)$ of (\ref{ACP}) is also not unique.
\end{rem}

\subsection{The existence of  solutions}
Now we consider the existence of such solutions of the Cauchy problem (\ref{ACP}).

\begin{prop}\label{58}
If $\mathbf{u}$ is a strict, resp. a $F$-solution, of the Cauchy problem (\ref{ACP}), then it is also strict weak, resp. $F$-weak solution, of (\ref{ACP}).
\end{prop}

\begin{proof} The proof follows from Lemma \ref{1295} and the fact that the mapping $\mathbf{g}\mapsto\langle\mathbf{g},\cdot\rangle$, $\tilde{{\mathcal D}}_{L^p,h}^{'s}(0,T;E)\rightarrow \mathcal{L}_b\left({\mathcal D}_{L^q,h}^s(0,T),E\right)$ is continuous.
\end{proof}

The proof of the next theorem heavily relies on the results
obtained in \cite{d81}. Parts in brackets are consequences of Proposition \ref{58}.

\begin{thm}\label{Fthm}
\begin{itemize}
\item[$i)$] The Cauchy problem (\ref{ACP}) has an $F$-solution (resp. an $F$-weak solution) in ${\tilde{{\mathcal D}}^{'s}_{L^p,h}(0,T;E)}$ for every $\mathbf{f}=(\mathbf{F}_{\alpha})_{\alpha}\in{\tilde{{\mathcal D}}^{'s}_{L^p,h}(0,T;E)}$ and $u_0=(u_{0,\alpha})_{\alpha}$ such that $(u_{0,\alpha})_{\alpha}$ satisfies (\ref{1290}) and $u_{0,\alpha}\in\overline{D(A)}$, $\forall \alpha\in\NN$. In the case of $F$-solution, it is unique.

\item[$ii)$] The Cauchy problem (\ref{ACP}) has a strict solution (resp. strict weak solution) in ${\tilde{{\mathcal D}}^{'s}_{L^p,h}(0,T;E)}$ for every $\mathbf{f}=(\mathbf{F}_{\alpha})_{\alpha}\in{\tilde{{\mathcal D}}^{'s}_{W^{1,p},h}(0,T;E)}$ and $u_0=(u_{0,\alpha})_{\alpha}$ such that $u_{0,\alpha}\in D(A)$ and $Au_{0,\alpha}+\mathbf{F}_{\alpha}(0)\in \overline{D(A)}$, $\forall\alpha\in\NN$ and $(u_{0,\alpha})_{\alpha}$ and $(Au_{0,\alpha})_{\alpha}$ satisfies (\ref{1290}). In the case of strict solution, it is unique.
\end{itemize}
\end{thm}

\begin{proof} First we will prove $i)$. By Theorem 7.2 of \cite{d81} (see also the Appendix of \cite{d81}) for each fixed $\alpha\in\NN$, the problem $\mathbf{U}'_{\alpha}=A\mathbf{U}_{\alpha}+\mathbf{F}_{\alpha}$, $\mathbf{U}_{\alpha}(0)=u_{0,\alpha}$ has a $F$-solution in $L^p(0,T;E)$. In other words, there exist $\mathbf{U}_{k,\alpha}\in W^{1,p}(0,T;E)\cap L^p(0,T;D(A))$, $\mathbf{F}_{k,\alpha}\in L^p(0,T;E)$, $u_{0,k,\alpha}\in E$, $k\in\ZZ_+$, such that $\mathbf{U}'_{k,\alpha}=A\mathbf{U}_{k,\alpha}+\mathbf{F}_{k,\alpha}$, $\mathbf{U}_{k,\alpha}(0)=u_{0,k,\alpha}$ and
\beq\label{970}
\lim_{k\rightarrow\infty}\Big(\|\mathbf{U}_{k,\alpha}-\mathbf{U}_{\alpha}\|_{L^p(0,T;E)}+ \|\mathbf{F}_{k,\alpha}-\mathbf{F}_{\alpha}\|_{L^p(0,T;E)}+\eeq \beqs\|u_{0,k,\alpha}-u_{0,\alpha}\|_E\Big)=0.
\eeqs
Moreover, by Theorem 5.1 of \cite{d81} (see also Theorem A.1 of the Appendix of \cite{d81}), each $\mathbf{U}_{\alpha}$ is in fact in $\mathcal{C}(0,T;E)$, $\mathbf{U}_{\alpha}(t)\in \overline{D(A)}$, $\forall t\in[0,T]$, $\mathbf{U}_{\alpha}(0)=u_{0,\alpha}$ and
\beq\label{975}
\|\mathbf{U}_{\alpha}(t)\|\leq C e^{\omega t}\left(\|\mathbf{U}_{\alpha}(0)\|+\int_0^t e^{-\omega s} \|\mathbf{F}_{\alpha}(s)\|ds\right),\, t\in[0,T].
\eeq
Using this estimate one easily verifies that $\mathbf{u}=(\mathbf{U}_{\alpha})_{\alpha}\in \tilde{{\mathcal D}}^{'s}_{L^p,h}(0,T;E)$. We will prove that this is an $(F)$-solution of (\ref{ACP}).\\
\indent Let $k\in\ZZ_+$. Take $n_k\in\ZZ_+$ such that
$$
\sum_{\alpha=n_k}^{\infty}\frac{\alpha!^{ps}}{h^{p\alpha}} \|\mathbf{F}_{\alpha}\|^p_{L^p(0,T;E)}\leq \frac{1}{(2k)^p},\, \sum_{\alpha=n_k}^{\infty}\frac{\alpha!^{ps}}{h^{p\alpha}} \|\mathbf{U}_{\alpha}\|^p_{L^p(0,T;E)}\leq \frac{1}{(2k)^p}$$
$$
\mbox{and }\sum_{\alpha=n_k}^{\infty}\frac{\alpha!^{ps}}{h^{p\alpha}} \|u_{0,\alpha}\|^p_E\leq \frac{1}{(2k)^p}.
$$
For each $0\leq \alpha\leq n_k-1$, by (\ref{970}) we can take $\mathbf{F}_{k_{\alpha},\alpha}$, $\mathbf{U}_{k_{\alpha},\alpha}$ and $u_{0,k_{\alpha},\alpha}$ such that
\beqs
\sum_{\alpha=0}^{n_k-1}\frac{\alpha!^{ps}}{h^{p\alpha}}\Big(\|\mathbf{U}_{k_{\alpha},\alpha}- \mathbf{U}_{\alpha}\|^p_{L^p(0,T;E)}&+& \|\mathbf{F}_{k_{\alpha},\alpha}-\mathbf{F}_{\alpha}\|^p_{L^p(0,T;E)}\\
&+&\|u_{0,k_{\alpha},\alpha}-u_{0,\alpha}\|^p_E\Big)\leq \frac{1}{(2k)^p}
\eeqs
and $\mathbf{U}'_{k_{\alpha},\alpha}=A\mathbf{U}_{k_{\alpha},\alpha}+\mathbf{F}_{k_{\alpha},\alpha}$, $\mathbf{U}_{k_{\alpha},\alpha}(0)=u_{0,k_{\alpha},\alpha}$. For $0\leq \alpha\leq n_k-1$ define $\mathbf{V}_{k,\alpha}=\mathbf{U}_{k_{\alpha},\alpha}$, $v_{0,k,\alpha}=u_{0,k_{\alpha},\alpha}$ and $\mathbf{G}_{k,\alpha}=\mathbf{F}_{k_{\alpha},\alpha}$. For $\alpha\geq n_k$ put $\mathbf{V}_{k,\alpha}=0$, $v_{0,k,\alpha}=0$ and $\mathbf{G}_{k,\alpha}=0$. Then $\mathbf{v}_k=(\mathbf{V}_{k,\alpha})_{\alpha}\in \tilde{{\mathcal D}}^{'s}_{W^{1,p},h}(0,T;E)\cap\tilde{{\mathcal D}}^{'s}_{L^p,h}(0,T;D(A))$, $\mathbf{g}_k=(\mathbf{G}_{k,\alpha})_{\alpha}\in \tilde{{\mathcal D}}^{'s}_{L^p,h}(0,T;E)$ and $v_{0,k}=(v_{0,k,\alpha})_{\alpha}$ is such that $\ds\sum_{\alpha=0}^{\infty}\frac{(\alpha!)^{ps}}{h^{p\alpha}}\|v_{0,k,\alpha}\|^p_E<\infty$. Also $\mathbf{v}_k(0)=v_{0,k}$. By definition, we have $\mathbf{V}'_{k,\alpha}=A\mathbf{V}_{k,\alpha}+\mathbf{G}_{k,\alpha}$ for all $\alpha\in\NN$. We will prove that $\mathbf{v}_k\rightarrow \mathbf{u}$, $\mathbf{g}_k\rightarrow \mathbf{f}$ and $v_{0,k}\rightarrow u_0$ in $\tilde{{\mathcal D}}^{'s}_{L^p,h}(0,T;E)$, hence $\mathbf{u}$ is $F$-solution of (\ref{ACP}). Let $\varepsilon>0$. Take $k_0\in \ZZ_+$ such that $1/k_0\leq \varepsilon$. For $k\geq k_0$, $k\in\ZZ_+$, we have\\
\\
$\|\mathbf{v}_k-\mathbf{u}\|^p_{\tilde{{\mathcal D}}^{'s}_{L^p,h}(0,T;E)}$
\beqs
&=&\sum_{\alpha=0}^{n_k-1}\frac{\alpha!^{ps}}{h^{p\alpha}}\|\mathbf{V}_{k,\alpha}- \mathbf{U}_{\alpha}\|^p_{L^p(0,T;E)}+\sum_{\alpha=n_k}^{\infty}\frac{\alpha!^{ps}}{h^{p\alpha}} \|\mathbf{U}_{\alpha}\|^p_{L^p(0,T;E)}\\
&\leq& \sum_{\alpha=0}^{n_k-1}\frac{\alpha!^{ps}}{h^{p\alpha}}\|\mathbf{U}_{k_{\alpha},\alpha}- \mathbf{U}_{\alpha}\|^p_{L^p(0,T;E)}+\frac{\varepsilon^p}{2^p}\leq \frac{2\varepsilon^p}{2^p}.
\eeqs
Hence $\|\mathbf{v}_k-\mathbf{u}\|_{\tilde{{\mathcal D}}^{'s}_{L^p,h}(0,T;E)}\leq \varepsilon$. Similarly, $\|\mathbf{g}_k-\mathbf{f}\|_{\tilde{{\mathcal D}}^{'s}_{L^p,h}(0,T;E)}\leq \varepsilon$ and $\ds\left(\sum_{\alpha=0}^{\infty}\frac{\alpha!^{ps}}{h^{p\alpha}}\|v_{0,k,\alpha}- u_{0,\alpha}\|^p_{L^p(0,T;E)}\right)^{1/p}\leq \varepsilon$, for $k\geq k_0$. It remains to prove the uniqueness. If $\tilde{\mathbf{u}}=(\tilde{\mathbf{U}}_{\alpha})_{\alpha}\in \tilde{{\mathcal D}}^{'s}_{L^p,h}(0,T;E)$ is another $F$-solution of (\ref{ACP}) then $\tilde{\mathbf{U}}_{\alpha}$ is a $F$-solution to the problem $\tilde{\mathbf{U}}'_{\alpha}(t)=A\tilde{\mathbf{U}}_{\alpha}(t)+\mathbf{F}_{\alpha}(t)$, $\tilde{\mathbf{U}}_{\alpha}(0)=u_{0,\alpha}$, for each $\alpha\in\NN$. But, theorem 5.1 of \cite{d81} (see also Theorem A.1 of the Appendix of \cite{d81}) implies that the $F$-solution to this problem must be unique, hence $\tilde{\mathbf{U}}_{\alpha}=\mathbf{U}_{\alpha}$ which proofs the desired uniqueness.\\
\indent To prove $ii)$, observe that Theorem 8.1 of \cite{d81} (see also Theorem A.2 of the Appendix of \cite{d81}) implies that for each $\alpha\in\NN$ there exists $\mathbf{U}_{\alpha}\in \mathcal{C}^1(0,T;E)\cap \mathcal{C}(0,T;D(A))$ such that
\beq\label{980}
\mathbf{U}'_{\alpha}(t)=A\mathbf{U}_{\alpha}(t)+\mathbf{F}_{\alpha}(t),\, \forall t\in[0,T] \mbox{ and } \mathbf{U}_{\alpha}(0)=u_{0,\alpha}
\eeq
and it satisfy (\ref{975}) and
\beq\label{987}
\|\mathbf{U}'_{\alpha}(t)\|\leq C e^{\omega t}\left(\|Au_{0,\alpha}+\mathbf{F}_{\alpha}(0)\|+\int_0^t e^{-\omega s} \|\mathbf{F}'_{\alpha}(s)\|ds\right),\, t\in[0,T].
\eeq
Moreover, by (\ref{980}) and (\ref{987}), we have
\beqs
\|A\mathbf{U}_{\alpha}(t)\|\leq C e^{2|\omega|T}\left(\|Au_{0,\alpha}\|+\|\mathbf{F}_{\alpha}(0)\|+T^{1/q} \|\mathbf{F}'_{\alpha}\|_{L^p(0,T;E)}\right)+\|\mathbf{F}_{\alpha}(t)\|,\, t\in[0,T].
\eeqs
Since $\mathbf{f}\in\tilde{{\mathcal D}}^{'s}_{W^{1,p},h}(0,T;E)$ and $(u_{0,\alpha})_{\alpha}$ and $(Au_{0,\alpha})_{\alpha}$ satisfy (\ref{1290}), by the above estimate and (\ref{975}) and (\ref{987}) we can conclude\\ $\mathbf{u}=(\mathbf{U}_{\alpha})_{\alpha}\in \tilde{{\mathcal D}}^{'s}_{W^{1,p},h}(0,T;E) \cap\tilde{{\mathcal D}}^{'s}_{L^p,h}(0,T;D(A))$. Hence $\mathbf{u}$ is a strict solution. The uniqueness follows from Theorem 8.1 of \cite{d81} (see also Theorem A.2 of the Appendix of \cite{d81}) by similar arguments as in $i)$.
\end{proof}

\subsection{Solutions in $\DD'^{(s)}_{L^1}(0,T;E)$}

Let $\mathbf{g}\in\DD'^{(s)}_{L^1}(0,T;E)$. By Theorem \ref{990} for $1<p<\infty$, there exists $h_1>0$ and $\mathbf{G}_{\alpha}\in L^p(0,T;E)$, $\alpha\in\NN$, such that
\beq\label{1100}
\sum_{\alpha=0}^{\infty}\frac{\alpha!^{ps}}{h_1^{p\alpha}}\|\mathbf{G}_{\alpha}\|^p_{L^p(0,T;E)}<\infty \mbox{ and } \mathbf{g}=\sum_{\alpha=0}^{\infty}\mathbf{G}^{(\alpha)}_{\alpha}.
\eeq
For the moment, for $\mathbf{g}\in \DD'^{(s)}_{L^1}(0,T;E)=\mathcal{L}_b\left(\dot{\BB}^{(s)}(0,T),E\right)$, denote by $\mathbf{g}(\varphi)$ the action of $\mathbf{g}$ on $\varphi\in\dot{\BB}^{(s)}(0,T)$. On the other hand, put $\tilde{\mathbf{g}}=(\mathbf{G}_{\alpha})_{\alpha}\in \tilde{{\mathcal D}}^{'s}_{L^p,h}(0,T;E)$. By the way we define the operator $\langle \tilde{\mathbf{g}},\cdot\rangle \in \mathcal{L}_b\left({\mathcal D}^{(s)}_{L^q,h}(0,T),E\right)$, one easily verifies that $\mathbf{g}(\varphi)=\langle \tilde{\mathbf{g}},\varphi\rangle$ for all $\varphi\in \dot{\BB}^{(s)}(0,T)\subseteq {\mathcal D}^{(s)}_{L^q,h}(0,T)$. Hence, if $\mathbf{g}\in\DD'^{(s)}_{L^1}(0,T;E)$ has the representation (\ref{1100}) we will denote by $\langle\mathbf{g},\cdot\rangle$ the action $\mathbf{g}(\cdot)$.\\
\indent Let $\mathbf{g}\in\DD'^{(s)}_{L^1}(0,T;E)$ has the representation (\ref{1100}). Define $\tilde{\mathbf{G}}_0=0$ and $\ds \tilde{\mathbf{G}}_{\alpha}(t)=\int_0^t \mathbf{G}_{\alpha-1}(s)ds$, $t\in[0,T]$ for $\alpha\in\ZZ_+$. Then, obviously, $\tilde{\mathbf{G}}_{\alpha}\in W^{1,p}(0,T;E)$, $\tilde{\mathbf{G}}_{\alpha}(0)=0$, $\tilde{\mathbf{G}}'_{\alpha}=\mathbf{G}_{\alpha-1}$ a.e. for all $\alpha\in\ZZ_+$, and if we put $h>h_1$ we have
\beq\label{1120}
\sum_{\alpha=0}^{\infty}\frac{\alpha!^{ps}}{h^{p\alpha}}\left(\|\tilde{\mathbf{G}}_{\alpha}\|^p_{L^p(0,T;E)}+ \|\tilde{\mathbf{G}}'_{\alpha}\|^p_{L^p(0,T;E)}\right)<\infty.
\eeq
By Theorem \ref{990}, $\sum_{\alpha=1}^{\infty} \tilde{\mathbf{G}}^{(\alpha)}_{\alpha}\in \DD'^{(s)}_{L^1}(0,T;E)$. Also, for $\varphi\in \dot{\BB}^{(s)}(0,T)$,
\beqs
\sum_{\alpha=1}^{\infty}(-1)^{\alpha}\int_0^T \tilde{\mathbf{G}}_{\alpha}(t)\varphi^{(\alpha)}(t)dt&=& \sum_{\alpha=0}^{\infty}(-1)^{\alpha}\int_0^T \tilde{\mathbf{G}}'_{\alpha+1}(t)\varphi^{(\alpha)}(t)dt\\
&=&\sum_{\alpha=0}^{\infty}(-1)^{\alpha}\int_0^T \mathbf{G}_{\alpha}(t)\varphi^{(\alpha)}(t)dt=\langle \mathbf{g},\varphi\rangle,
\eeqs
i.e. $\mathbf{g}=\sum_{\alpha=1}^{\infty} \tilde{\mathbf{G}}^{(\alpha)}_{\alpha}$. In other words, for $\mathbf{g}\in \DD'^{(s)}_{L^1}(0,T;E)$ and $1<p<\infty$ we can always find $h>0$ such that $\mathbf{g}=\sum_{\alpha} \tilde{\mathbf{G}}^{(\alpha)}_{\alpha}$, where $\tilde{\mathbf{G}}_{\alpha}\in W^{1,p}(0,T;E)$, $\tilde{\mathbf{G}}_{\alpha}(0)=0$, $\alpha\in\NN$, such that (\ref{1120}) holds. Moreover, in this notation, if we put $\tilde{\mathbf{f}}=(\tilde{\mathbf{G}}'_{\alpha})_{\alpha}\in \tilde{{\mathcal D}}^{'s}_{L^p,h}(0,T;E)$, then $\langle\tilde{\mathbf{f}},\cdot\rangle$ and the $E$-valued ultradistribution $\mathbf{g}'\in \DD'^{(s)}_{L^1}(0,T;E)$ (where $\mathbf{g}'$ is the ultradistributional derivative of $\mathbf{g}$) generate the same element in $\DD'^{(s)}_{L^1}(0,T;E)\cong\mathcal{L}_b\left(\dot{\BB}^{(s)}(0,T),E\right)$. To see this, for $\varphi\in \dot{\BB}^{(s)}(0,T)$ we calculate as follows
\beqs
\langle\tilde{\mathbf{f}},\varphi\rangle=\sum_{\alpha=0}^{\infty}(-1)^{\alpha}\int_0^T \tilde{\mathbf{G}}'_{\alpha}(t)\varphi^{(\alpha)}(t)dt=-\sum_{\alpha=0}^{\infty}(-1)^{\alpha}\int_0^T \tilde{\mathbf{G}}_{\alpha}(t)\varphi^{(\alpha+1)}(t)dt
\eeqs
which is exactly the value at $\varphi$ of the ultradistributional derivative of $\mathbf{g}\in \DD'^{(s)}_{L^1}(0,T;E)$.\\
\indent We consider the equation $\mathbf{u}'=A\mathbf{u}+\mathbf{f}$ in $\DD'^{(s)}_{L^1}(0,T;E)$. In other words, $\mathbf{f}\in \DD'^{(s)}_{L^1}(0,T;E)$ is given, we search for $\mathbf{u}\in \DD'^{(s)}_{L^1}(0,T;E)$ such that, for every $\varphi\in \dot{\BB}^{(s)}(0,T)$, $\langle \mathbf{u},\varphi\rangle \in D(A)$ and $\langle\mathbf{u}',\varphi\rangle=A\langle\mathbf{u},\varphi\rangle+\langle\mathbf{f},\varphi\rangle$. By the above discussion, for $1<p<\infty$, there exists $h>0$ and $\mathbf{F}_{\alpha}\in W^{1,p}(0,T;E)$, $\mathbf{F}_{\alpha}(0)=0$, $\alpha\in\NN$, such that (\ref{1120}) holds (with $\mathbf{F}_{\alpha}$ and $\mathbf{F}'_{\alpha}$ in place of $\tilde{\mathbf{G}}_{\alpha}$ and $\tilde{\mathbf{G}}'_{\alpha}$) and $\mathbf{f}=\sum_{\alpha=0}^{\infty}\mathbf{F}^{(\alpha)}_{\alpha}$. If we put $\tilde{\mathbf{f}}=(\mathbf{F}_{\alpha})_{\alpha}$, then $\tilde{\mathbf{f}}\in \tilde{{\mathcal D}}^{'s}_{W^{1,p},h}(0,T;E)$. For $u_{0,\alpha}=0\in D(A)$ put $u_0=(u_{0,\alpha})_{\alpha}$. Then the conditions of Theorem \ref{Fthm} $ii)$ are satisfied, hence there exists $\tilde{\mathbf{u}}=(\mathbf{U}_{\alpha})_{\alpha}\in \tilde{{\mathcal D}}^{'s}_{W^{1,p},h}(0,T;E)\cap \tilde{{\mathcal D}}^{'s}_{L^p,h}(0,T;D(A))$ which is a strict weak solution of $\tilde{\mathbf{u}}'=A\tilde{\mathbf{u}}+\tilde{\mathbf{f}}$ in $\tilde{{\mathcal D}}^{'s}_{L^p,h}(0,T;E)$. If we put $\mathbf{u}=\sum_{\alpha=0}^{\infty} \mathbf{U}^{(\alpha)}_{\alpha}\in \DD'^{(s)}_{L^1}(0,T;E)$, by the above discussion, $\langle \mathbf{u},\varphi\rangle\in D(A)$, $\forall \varphi\in \dot{\BB}^{(s)}(0,T)$ (since this holds for $\tilde{\mathbf{u}}$) and $\mathbf{u}$ is a solution of $\mathbf{u}'=A\mathbf{u}+\mathbf{f}$ in $\DD'^{(s)}_{L^1}(0,T;E)$. Moreover, by Theorem \ref{350} this $\mathbf{u}$ as well as $\mathbf{f}$ are in fact elements of $\DD'^{(s)}_{L^p,h}(0,T;E)$. Thus, we proved the following theorem.

\begin{thm}\label{1220}
Let $A:D(A)\subseteq E\rightarrow E$ be a closed operator which satisfies the Hille-Yosida condition and $\mathbf{f}\in \DD'^{(s)}_{L^1}(0,T;E)$. Then the equation $\mathbf{u}'=A\mathbf{u}+\mathbf{f}$ always has a solution $\mathbf{u}\in\DD'^{(s)}_{L^1}(0,T;E)$. Moreover, $\mathbf{u}\in\DD'^{(s)}_{L^p,h}(0,T;E)$ where $1<p<\infty$ and $h>0$ are such that
\beqs
\sum_{\alpha=0}^{\infty}\frac{\alpha!^{ps}}{h^{p\alpha}}\left(\|\mathbf{F}_{\alpha}\|^p_{L^p(0,T;E)}+ \|\mathbf{F}'_{\alpha}\|^p_{L^p(0,T;E)}\right)<\infty,
\eeqs
with $\mathbf{f}=\sum_{\alpha}\mathbf{F}^{(\alpha)}_{\alpha}$, where $\mathbf{F}_{\alpha}\in W^{1,p}(0,T;E)$, $\mathbf{F}_{\alpha}(0)=0$, $\alpha\in\NN$.
\end{thm}

\section{Applications}

Theorem \ref{1220} is applicable in variety of different
situations. We collect some of them in the next proposition. First
we need the following definition given in \cite{Pazy}.

\begin{defn}
Let $\Omega$ be bounded open domain with smooth boundary in $\RR^d$ and $m\in\ZZ_+$. We say that $A(x,\partial_x)=\sum_{|\alpha|\leq 2m} a_{\alpha}(x)\partial^{\alpha}_x$ where $a_{\alpha}\in\mathcal{C}^{2m}(\overline{\Omega})$, is strongly elliptic if there exists $c>0$ such that
\beqs
\mathrm{Re}(-1)^m\sum_{|\alpha|=2m}a_{\alpha}(x)\xi^{\alpha}\geq c|\xi|^{2m},\, \forall x\in\overline{U},\, \forall \xi\in\RR^d.
\eeqs
\end{defn}

\begin{prop}
The operator $A:D(A)\subseteq E\rightarrow E$ is closed operator which satisfies the Hille-Yosida condition in each of the following situations:
\begin{itemize}
\item[$i)$] (\cite{d81}) $E=\mathcal{C}([0,1])$, $Av=-v'$, $D(A)=\{v\in\mathcal{C}^1([0,1])|\, v(0)=0\}$;
\item[$ii)$] (\cite{d81}) for $\kappa\in(0,1)$, $E=\mathcal{C}^{\kappa}_0([0,1])=\{v\in\mathcal{C}^{\kappa}([0,1])|\, v(0)=0\}$, $Av=-v'$, $D(A)=\{v\in\mathcal{C}^{1+\kappa}([0,1])|\, v(0)=v'(0)=0\}$;
\item[$iii)$] (\cite{d81}) $E=\mathcal{C}([0,1])$, $Av=v''$, $D(A)=\{v\in\mathcal{C}^2([0,1])|\, v(0)=v(1)=0\}$;
\item[$iv)$] (\cite{d81}) for $\Omega$ bounded open set with regular boundary in $\RR^d$, $E=\mathcal{C}(\overline{\Omega})$, $Av=\Delta v$, $D(A)=\{v\in\mathcal{C}(\overline{\Omega})|\, v_{| \partial \Omega}=0,\, \Delta v\in \mathcal{C}(\overline{\Omega})\}$ (here $\Delta$ is the Laplacian in the sense of distributions in $\Omega$);
\item[$v)$] (\cite{Pazy}) let $\Omega$ be bounded open domain with smooth boundary in $\RR^d$ and $m\in\ZZ_+$. Let $A(x,\partial_x)$ be strongly elliptic. Define $E=L^p(\Omega)$, $Av=-A(x,\partial_x)v$, $D(A)=W^{2m,p}(\Omega)\cap W^{m,p}_0(\Omega)$, for $1<p<\infty$ and for $p=1$ define $E=L^1(\Omega)$, $Av=-A(x,\partial_x)v$, $D(A)=\{v\in W^{2m-1,1}(\Omega)\cap W^{m,1}_0(\Omega)|\, A(x,\partial_x)v\in L^1(\Omega)\}$.
\end{itemize}
In particular, for $\mathbf{f}\in\DD'^{(s)}_{L^1}(0,T;E)$, the equation $\mathbf{u}'_t= A\mathbf{u}+\mathbf{f}$ always has solution in $\DD'^{(s)}_{L^1}(0,T;E)$.
\end{prop}

\begin{proof} The facts that $A:D(A)\subseteq E\rightarrow E$ is closed operator which satisfies the Hille-Yosida condition when $A$ and $E$ are defined as in $i)-iv)$ are proven in Section 14 of \cite{d81}. When $A$ and $E$ are defined as in $v)$ Theorem 7.3.5, pg. 214, of \cite{Pazy} for the case $1<p<\infty$, resp. Theorem 7.3.10, pg. 218, of \cite{Pazy} for the case $p=1$, implies that $A$ is closed operator which satisfies the Hille-Yosida condition (in fact these theorems state that $A$ is the infinitesimal generator of analytic semigroup on $L^p(\Omega)$, $1\leq p<\infty$). Now, the fact that the equation $\mathbf{u}'_t= A\mathbf{u}+\mathbf{f}$ has solution in $\DD'^{(s)}_{L^1}(0,T;E)$ follows from Theorem \ref{1220}.
\end{proof}

\subsection{Parabolic equation in $\DD'^{(s)}_{L^1}(U)$}

In this subsection $U$ is a bounded domain in $\RR^d$ with smooth boundary. For the brevity in notation, let $\tilde{{\mathcal D}}^{'s}_{L^p,h}(U)$, resp. $\tilde{{\mathcal D}}^{'s}_{W^{1,p},h}(U)$, be the space $\tilde{{\mathcal D}}^{'s}_{L^p,h}(0,T;E)$, resp. $\tilde{{\mathcal D}}^{'s}_{W^{1,p},h}(0,T;E)$, when $E=\CC$. Also, for $k\in\ZZ_+$, by $\tilde{{\mathcal D}}^{'s}_{W^{k,p},h}(U)$ we denote the space of all sequences $(F_{\alpha})_{\alpha}$, $F_{\alpha}\in W^{k,p}(U)$, $\forall \alpha\in\NN^d$, for which
\beqs
\left\|(F_{\alpha})_{\alpha}\right\|_{\tilde{{\mathcal D}}^{'s}_{W^{k,p},h}(U)}=\left(\sum_{\alpha\in\NN^d}\frac{{\alpha}!^{ps}}{h^{p\alpha}} \|F_{\alpha}\|^p_{W^{k,p}(U)}\right)^{1/p}<\infty.
\eeqs
It is easy to verify that it becomes a $(B)$-space with the norm $\|\cdot\|_{\tilde{{\mathcal D}}^{'s}_{W^{k,p},h}(U)}$.\\
\indent Let $m\in\ZZ_+$, $A(x,\partial_x)=\sum_{|\alpha|\leq 2m} a_{\alpha}(x)\partial^{\alpha}_x$, where $a_{\alpha}\in\EE^{(s)}(V)$ for some open set $V\subseteq \RR^d$ and $U\subset\subset V$. We assume that $A(x,\partial_x)$ is a strongly elliptic operator. Obviously, $A(x,\partial_x)$ is continuous operator on $\dot{\BB}^{(s)}(U)$ and on $\DD'^{(s)}_{L^1}(U)$. Denote by $\tilde{A}:D(\tilde{A})\subseteq L^2(U)\rightarrow L^2(U)$ the following unbounded operator
\beqs
D(\tilde{A})=W^{2m,2}(U)\cap W^{m,2}_0(U),\quad \tilde{A}(\varphi)=A(x,\partial_x)\varphi,\, \varphi\in D(\tilde{A}).
\eeqs
For such $A(x,\partial_x)$ the following a priori estimate holds (see Theorem 7.3.1, pg. 212, of \cite{Pazy}).

\begin{prop}\label{1205}\cite{Pazy}
Let $A(x,\partial_x)$ be strongly elliptic operator of order $2m$ on a bounded domain $U$ with smooth boundary $\partial U$ in $\RR^d$ and let $1<p<\infty$. Then, there exists a constant $\tilde{C}>0$ such that
\beqs
\|\varphi\|_{W^{2m,p}(U)}\leq \tilde{C}\left(\|A(x,\partial_x)\varphi\|_{L^p(U)}+\|\varphi\|_{L^p(U)}\right),\, \forall \varphi\in W^{2m,p}(U)\cap W^{m,p}_0(U).
\eeqs
\end{prop}

Moreover, Theorem 7.3.5, pg. 214, of \cite{Pazy}, yields that
$-\tilde{A}$ is the infinitesimal generator of an analytic
semigroup of operators on $L^2(U)$. In particular $-\tilde{A}$ is
closed and it satisfies the Hille-Yosida condition (\ref{1410}) for some $\omega,C>0$.\\
Now we can
prove the theorem announced in the introductions. Note that we
need to prove the theorem for $\DD'^{(s)}_{L^1}\left((0,T)\times
U\right)$, since $\DD'^{(s)}_{L^p}\left((0,T)\times U\right)$ and
$\DD'^{(s)}_{L^1}\left((0,T)\times U\right)$ are isomorphic l.c.s.

\begin{thm}
Let $U$ be a bounded domain in $\RR^d$ with smooth boundary and
$A(x,\partial_x)$ strongly elliptic operator of order $2m$ on $U$.
Then for each $f\in \DD'^{(s)}_{L^1}\left((0,T)\times U\right)$
there exists $u\in \DD'^{(s)}_{L^1}\left((0,T)\times U\right)$
such that $u'_t+A(x,\partial_x)u=f$ in
$\DD'^{(s)}_{L^1}\left((0,T)\times U\right)$.
\end{thm}

\begin{proof} Denote by $A$
the following unbounded operator: \beqs
A\tilde{f}&=&(-A(x,\partial_x)F_{\alpha})_{\alpha}\left(=(-\tilde{A}F_{\alpha})_{\alpha}\right),\\
D(A)&=&\left\{\tilde{f}=(F_{\alpha})_{\alpha}\in \tilde{{\mathcal D}}^{'s}_{W^{2m,2},h}(U)|\, F_{\alpha}\in W^{m,2}_0(U),\, \forall \alpha\in \NN^d\right\}.\eeqs
Then, obviously, $A:D(A)\subseteq \tilde{{\mathcal D}}^{'s}_{L^2,h}(U)\rightarrow \tilde{{\mathcal D}}^{'s}_{L^2,h}(U)$ is a linear operator. Since $\tilde{A}$ is closed, by Proposition \ref{1205}, it is easy to verify that $A$ is closed. For $\lambda>\omega$, define $B_{\lambda}:\tilde{{\mathcal D}}^{'s}_{L^2,h}(U)\rightarrow \tilde{{\mathcal D}}^{'s}_{L^2,h}(U)$, by $B_{\lambda}(\tilde{f})=(R(\lambda: -\tilde{A})F_{\alpha})_{\alpha}$. For $\tilde{f}=(F_{\alpha})_{\alpha}\in \tilde{{\mathcal D}}^{'s}_{L^2,h}(U)$,
\beqs
\|B_{\lambda}\tilde{f}\|_{\tilde{{\mathcal D}}^{'s}_{L^2,h}(U)}= \left(\sum_{|\alpha|=0}^{\infty}\frac{\alpha!^{2s}}{h^{2|\alpha|}}\|R(\lambda: -\tilde{A})F_{\alpha}\|^2_{L^2(U)}\right)^{1/2}\leq \frac{C}{\lambda-\omega}\|\tilde{f}\|_{\tilde{{\mathcal D}}^{'s}_{L^2,h}(U)}.
\eeqs
Hence $B_{\lambda}$ is well defined continuous operator. For $(F_{\alpha})_{\alpha}\in \tilde{{\mathcal D}}^{'s}_{L^2,h}(U)$, by the Hille-Yosida condition for $-\tilde{A}$, Proposition \ref{1205} and the fact that $\tilde{A}R(\lambda:-\tilde{A})=\mathrm{Id}-\lambda R(\lambda:-\tilde{A})$, we obtain
\beqs
\left\|R(\lambda:-\tilde{A})F_{\alpha}\right\|_{W^{2m,2}(U)}\leq \tilde{C}\left(1+\frac{C(\lambda+1)}{\lambda-\omega}\right)\|F_{\alpha}\|_{L^2(U)}.
\eeqs
This implies that $B_{\lambda}(F_{\alpha})_{\alpha}=(R(\lambda: -\tilde{A})F_{\alpha})_{\alpha}\in \tilde{{\mathcal D}}^{'s}_{W^{2m,2},h}(U)$. Obviously $R(\lambda: -\tilde{A})F_{\alpha}\in W^{m,2}_0(U)$, for each $\alpha\in\NN^d$. Hence, the image of $B_{\lambda}$ is contained in $D(A)$. Conversely, for $(F_{\alpha})_{\alpha}\in D(A)$, let $G_{\alpha}=(\lambda+\tilde{A})F_{\alpha}$, for each $\alpha\in\NN^d$. Then $(G_{\alpha})_{\alpha}\in \tilde{{\mathcal D}}^{'s}_{L^2,h}(U)$ and $B_{\lambda}(G_{\alpha})_{\alpha}=(F_{\alpha})_{\alpha}$. Hence, the image of $B_{\lambda}$ is $D(A)$. Also, $(\lambda-A)B_{\lambda}=\mathrm{Id}$ and $B_{\lambda}(\lambda-A)=\mathrm{Id}$. We obtain that $\lambda>\omega$ is in the resolvent of $A$, $R(\lambda:A)=B_{\lambda}$, and similarly as above one can prove that $\|(\lambda-\omega)^k R(\lambda:A)^k\|_{\mathcal{L}\left(\tilde{{\mathcal D}}^{'s}_{L^2,h}(U)\right)}\leq C$, i.e. $A$ satisfies the Hille-Yosida condition.\\
\indent We want to solve the equation $u'_t(t,x)+A(x,\partial_x)u(t,x)=f(t,x)$ in $\DD'^{(s)}_{L^1}\left((0,T)\times U\right)$. For the simplicity of notation put $U_1=(0,T)\times U$. By Proposition \ref{330}, there exist $h>0$ and $F_{\alpha,\beta}(t,x)\in \mathcal{C}\left(\overline{U_1}\right)$, $\alpha\in\NN$, $\beta\in\NN^d$ such that
\beq\label{1210}
f=\sum_{\alpha,\beta}\partial^{\alpha}_t \partial^{\beta}_x F_{\alpha,\beta}\,\,\,\, \mbox{and}\,\,\,\, \sum_{\alpha,\beta}\frac{\left(\alpha!\beta!\right)^{2s}}{h^{2(\alpha+|\beta|)}}\|F_{\alpha,\beta}\|^2_{L^{\infty}(\overline{U_1})} <\infty.
\eeq
Let $E=\tilde{{\mathcal D}}^{'s}_{L^2,h}(U)$. Let $\ds C'_1=1+\sup_{\beta\in\NN^d}h^{|\beta|}/\beta!^s$ and put $C_1=(1+T+|U|)C'_1$. Let $L_f$ be the mapping $\varphi\mapsto L_f(\varphi)$, $\dot{\BB}^{(s)}(0,T)\rightarrow E$ defined by $L_f(\varphi)=(\tilde{F}_{\varphi,\beta})_{\beta}$, where $\ds \tilde{F}_{\varphi,\beta}(x)=\sum_{\alpha}(-1)^{\alpha}\int_0^T F_{\alpha,\beta}(t,x)\varphi^{(\alpha)}(t)dt$. We prove that it is well defined and continuous mapping. First we prove that $\tilde{F}_{\varphi,\beta}$ is continuous function on $\overline{U}$ for each $\beta\in\NN^d$ and $\varphi\in\dot{\BB}^{(s)}(0,T)$. For $\varepsilon>0$, by (\ref{1210}), we can find $k_0\in\ZZ_+$ such that $\ds \sum_{\alpha+|\beta|\geq k_0} \frac{\left(\alpha!\beta!\right)^{2s}}{h^{2(\alpha+|\beta|)}}\|F_{\alpha,\beta}\|^2_{L^{\infty}(\overline{U_1})} <\frac{\varepsilon^2}{(4C_1)^2}$. For each $\alpha\in\NN$, $\beta\in\NN^d$, $F_{\alpha,\beta}$ is uniformly continuous (since $\overline{U_1}$ is compact in $\RR^{d+1}$), hence there exists $\delta>0$ such that for every $t,t'\in[0,T]$, $x,x'\in \overline{U}$ such that $|t-t'|\leq\delta$ and $|x-x'|\leq \delta$,
\beqs
\sum_{\alpha+|\beta|=0}^{k_0-1}\frac{\left(\alpha!\beta!\right)^{2s}}{h^{2(\alpha+|\beta|)}} \left|F_{\alpha,\beta}(t,x)-F_{\alpha,\beta}(t',x')\right|^2 <\frac{\varepsilon^2}{(2C_1)^2}.
\eeqs
Hence\\
\\
$\left|\tilde{F}_{\varphi,\beta}(x)-\tilde{F}_{\varphi,\beta}(x')\right|$
\beqs
\leq \|\varphi\|_{\DD^{(s)}_{L^2,h}(0,T)} \left(\sum_{\alpha=0}^{\infty}\frac{(\alpha!)^{2s}}{h^{2\alpha}}\int_0^T \left|F_{\alpha,\beta}(t,x)-F_{\alpha,\beta}(t,x')\right|^2dt\right)^{1/2}\leq\eeqs \beqs \leq \varepsilon\|\varphi\|_{\DD^{(s)}_{L^2,h}(0,T)}
\eeqs
and the continuity of $\tilde{F}_{\varphi,\beta}$ follows. Also, one easily verifies that
\beqs
\left(\sum_{\beta}\frac{\beta!^{2s}}{h^{2|\beta|}}\left\|\tilde{F}_{\varphi,\beta}\right\|^2_{L^{\infty}(U)}\right)^{1/2}\leq\eeqs  \beqs T^{1/2}\|\varphi\|_{\DD^{(s)}_{L^2,h}(U)} \left(\sum_{\alpha,\beta}\frac{\left(\alpha!\beta!\right)^{2s}}{h^{2(\alpha+|\beta|)}} \|F_{\alpha,\beta}\|^2_{L^{\infty}(\overline{U_1})}\right)^{1/2}.
\eeqs
Since $\left\|\tilde{F}_{\varphi,\beta}\right\|_{L^2(U)}\leq |U|^{1/2}\left\|\tilde{F}_{\varphi,\beta}\right\|_{L^{\infty}(U)}$, we obtain that $L_f$ is well defined and $L_f\in\mathcal{L}\left(\dot{\BB}^{(s)}(0,T),E\right)$. Now, as $\mathcal{L}_b\left(\dot{\BB}^{(s)}(0,T),E\right)\cong\DD'^{(s)}_{L^1}(0,T;E)$ denote by $\mathbf{f}\in\DD'^{(s)}_{L^1}(0,T;E)$ the mapping $L_f$.\\
\indent Now, Theorem \ref{1220} implies that there exists $\mathbf{u}\in \DD'^{(s)}_{L^1}(0,T;E)$ such that $\mathbf{u}'=A\mathbf{u}+\mathbf{f}$ in $\DD'^{(s)}_{L^1}(0,T;E)$. Each element $\mathbf{g}=(G_{\alpha})_{\alpha}\in E=\tilde{{\mathcal D}}^{'s}_{L^p,h}(U)$ generates an element of $\mathcal{L}_b\left(\dot{\BB}^{(s)}(U),\CC\right)=\DD'^{(s)}_{L^1}(U)$ (see Section \ref{910}) by $\ds \langle S(\mathbf{g}),\psi\rangle=\sum_{\beta}(-1)^{|\beta|}\int_U G_{\beta}(x)\partial^{\beta}_x\psi(x)dx$ and one easily verifies that the mapping $S:E\rightarrow \DD'^{(s)}_{L^1}(U)$, $\mathbf{g}\mapsto S(\mathbf{g})$, is continuous. Hence, we have the continuous mapping $\varphi\mapsto S(\langle \mathbf{u},\varphi\rangle)$, given by
\beqs
\dot{\BB}^{(s)}(0,T)\xrightarrow{\langle \mathbf{u},\cdot\rangle}E\xrightarrow{S}\DD'^{(s)}_{L^1}(U).
\eeqs
Since $\varphi\mapsto S(\langle \mathbf{u},\varphi\rangle)\in\mathcal{L}_b\left(\dot{\BB}^{(s)}(0,T),\DD'^{(s)}_{L^1}(U)\right)\cong\DD'^{(s)}_{L^1}(U_1)$ (where the isomorphism follows from Theorem \ref{1310}), denote by $u\in\DD'^{(s)}_{L^1}(U_1)$ this ultradistribution. Then, for $\varphi\in \dot{\BB}^{(s)}(0,T)$, $\psi\in \dot{\BB}^{(s)}(U)$, $\langle u(t,x),\varphi(t)\psi(x)\rangle=\langle S(\langle \mathbf{u},\varphi\rangle),\psi\rangle$. Since $\langle \mathbf{u}',\varphi\rangle=-\langle \mathbf{u},\varphi'\rangle$, for all $\varphi\in \dot{\BB}^{(s)}(0,T)$ we have \\ $\langle u'_t(t,x),\varphi(t)\psi(x)\rangle=-\langle u(t,x),\varphi'(t)\psi(x)\rangle=\langle S(\langle \mathbf{u}',\varphi\rangle),\psi\rangle$, for all $\varphi\in \dot{\BB}^{(s)}(0,T)$, $\psi\in \dot{\BB}^{(s)}(U)$. Also, for $\varphi\in \dot{\BB}^{(s)}(0,T)$, since $\langle\mathbf{u},\varphi\rangle\in D(A)$, $\langle \mathbf{u},\varphi\rangle=(G_{\varphi,\beta})_{\beta}\in D(A)$. Then, by the definition of $A$, $A\langle \mathbf{u},\varphi\rangle=(-\tilde{A}G_{\varphi,\beta})_{\beta}\in E$. Now, for $\psi\in\dot{\BB}^{(s)}(U)$,\\
\\
$\ds\left\langle S\left((-\tilde{A}G_{\varphi,\beta})_{\beta}\right),\psi\right\rangle$
\beqs
&=&-\sum_{\beta}(-1)^{|\beta|}\int_U \tilde{A}G_{\varphi,\beta}(x)\partial^{\beta}_x\psi(x)dx\\
&=&-\sum_{\beta}(-1)^{|\beta|}\int_U G_{\varphi,\beta}(x){}^tA(x,\partial_x)\partial^{\beta}_x\psi(x)dx=-\langle S(\langle \mathbf{u},\varphi\rangle),{}^tA(x,\partial_x)\psi\rangle\\
&=&-\langle u(t,x),\varphi(t){}^tA(x,\partial_x)\psi(x)\rangle=-\langle A(x,\partial_x)u(t,x),\varphi(t)\psi(x)\rangle,
\eeqs
i.e. $\left\langle S\left(A\langle \mathbf{u},\varphi\rangle\right),\psi\right\rangle=-\langle A(x,\partial_x)u(t,x),\varphi(t)\psi(x)\rangle$ for all $\varphi\in \dot{\BB}^{(s)}(0,T)$, $\psi\in \dot{\BB}^{(s)}(U)$. Moreover, observe that for $\varphi\in \dot{\BB}^{(s)}(0,T)$, $\psi\in \dot{\BB}^{(s)}(U)$, we have
\beqs
\left\langle S(\langle \mathbf{f},\varphi\rangle),\psi\right\rangle&=&\sum_{\beta}(-1)^{|\beta|}\int_U \tilde{F}_{\varphi,\beta}(x)\partial^{\beta}_x\psi(x)dx\\
&=&\sum_{\alpha,\beta}(-1)^{\alpha+|\beta|}\int_{U_1} F_{\alpha,\beta}(t,x)\varphi^{(\alpha)}(t)\partial^{\beta}_x\psi(x)dtdx\\
&=&\langle f(t,x),\varphi(t)\psi(x)\rangle,
\eeqs
where, in the second equality, we used the definition of $\tilde{F}_{\varphi,\beta}$ and Fubini's theorem since $\ds\sum_{\alpha,\beta}\int_{U_1} \left|F_{\alpha,\beta}(t,x)\right|\left|\varphi^{(\alpha)}(t)\right|\left|\psi^{(\beta)}(x)\right|dtdx<\infty$ by (\ref{1210}). Now, since $\mathbf{u}'=A\mathbf{u}+\mathbf{f}$ in $\DD'^{(s)}_{L^1}(0,T;E)$, for every $\varphi\in \dot{\BB}^{(s)}(0,T)$, $\langle\mathbf{u}'(t),\varphi(t)\rangle=A\langle\mathbf{u}(t),\varphi(t)\rangle+\langle\mathbf{f}(t),\varphi(t)\rangle$ in $E$. Then $S\left(\langle\mathbf{u}',\varphi\rangle\right)=S\left(A\langle\mathbf{u},\varphi\rangle\right)+ S\left(\langle\mathbf{f},\varphi\rangle\right)$ in $\DD'^{(s)}_{L^1}(U)$. Hence, for $\varphi\in \dot{\BB}^{(s)}(0,T)$, $\psi\in \dot{\BB}^{(s)}(U)$, we have
\beqs
\langle u'_t(t,x),\varphi(t)\psi(x)\rangle&=&\langle S(\langle \mathbf{u}',\varphi\rangle),\psi\rangle=\left\langle S\left(A\langle\mathbf{u},\varphi\rangle\right),\psi\right\rangle+ \left\langle S\left(\langle\mathbf{f},\varphi\rangle\right),\psi\right\rangle\\
&=&-\langle A(x,\partial_x)u(t,x),\varphi(t)\psi(x)\rangle+\langle
f(t,x),\varphi(t)\psi(x)\rangle. \eeqs Since
$\dot{\BB}^{(s)}(0,T)\hat{\otimes}\dot{\BB}^{(s)}(U)\cong\dot{\BB}^{(s)}(U_1)$
by Theorem \ref{1310}, we obtain the claim in the theorem.
\end{proof}

\begin{ex}
An interesting application of this theorem is obtained by taking $A(x,\partial_x)$ to be $-\Delta_x$ ($\Delta_x$ is the Laplacian $\partial_{x_1}^2+...+\partial_{x_d}^2$) and $U$ to be arbitrary bounded domain with smooth boundary in $\RR^d$. Then $-\Delta_x$ is strongly elliptic operator of order $2$ on $U$. The above theorem then asserts that for $f\in \DD'^{(s)}_{L^1}\left((0,T)\times U\right)$ the equation $u'_t-\Delta_x u=f$ always has solution in $\DD'^{(s)}_{L^1}\left((0,T)\times U\right)$.\end{ex}
\begin{ex}
If $U=(0,T_1)\subseteq\RR$ and $A$ is differentiation in $x$, arguing as above, one can prove the following assertion:
Let $f\in \DD'^{(s)}_{L^1}\left((0,T)\times (0,T_1)\right)$. The equation $u'_t+u'_x=f$ always has a solution in $\DD'^{(s)}_{L^1}\left((0,T)\times (0,T_1)\right)$.
\end{ex}

\end{document}